\documentclass[12pt,reqno]{amsart}
\usepackage{amsmath,amssymb,amsthm,mathtools,calc,verbatim,enumitem,tikz,url,hyperref,mathrsfs,cite,fullpage} 

\newtheorem{theorem}{Theorem}[section]
\newtheorem{lemma}[theorem]{Lemma}
\newtheorem{corollary}[theorem]{Corollary}
\newtheorem{conj}[theorem]{Conjecture}

\newtheorem{claim}{Claim}
\theoremstyle{definition}

\newtheorem*{qu*}{Question}
\theoremstyle{remark}

\newcommand\N{\mathbb{N}}
\newcommand\R{\mathbb{R}}
\newcommand\Z{\mathbb{Z}}
\newcommand\E{\operatorname{\mathbb{E}}}
\newcommand\cA{\mathcal{A}}
\newcommand\cB{\mathcal{B}}

\newcommand\cS{\mathcal{S}}

\renewcommand\Pr{\operatorname{\mathbb{P}}}
\newcommand\id{\hbox{$1\mkern-6.5mu1$}}
\newcommand\lcm{\operatorname{lcm}}
\newcommand\eps{\varepsilon}

\renewcommand\le{\leqslant}
\renewcommand\ge{\geqslant}
\renewcommand\to{\rightarrow}
\newcommand\ds{\displaystyle}
\newcommand\bmid{\mathrel{\big|}}

\begin{document}

\title{On the Erd\H{o}s Covering Problem: \\ The density of the uncovered set}
\author{Paul Balister \and B\'ela Bollob\'as \and Robert Morris \and \\
Julian Sahasrabudhe \and Marius Tiba}

\address{Department of Mathematical Sciences,
University of Memphis, Memphis, TN 38152, USA}\email{pbalistr@memphis.edu}

\address{Department of Pure Mathematics and Mathematical Statistics, Wilberforce Road,
Cambridge, CB3 0WA, UK, and Department of Mathematical Sciences,
University of Memphis, Memphis, TN 38152, USA}\email{b.bollobas@dpmms.cam.ac.uk}

\address{IMPA, Estrada Dona Castorina 110, Jardim Bot\^anico,
Rio de Janeiro, 22460-320, Brazil}\email{rob@impa.br}

\address{Peterhouse, Trumpington Street, University of Cambridge, CB2 1RD, UK \and IMPA, Estrada Dona Castorina 110, Jardim Bot\^anico,
Rio de Janeiro, 22460-320, Brazil }\email{julians@impa.br}

\address{Department of Pure Mathematics and Mathematical Statistics, Wilberforce Road,
Cambridge, CB3 0WA, UK}\email{mt576@dpmms.cam.ac.uk}

\thanks{The first two authors were partially supported by NSF grant DMS 1600742, the third author was partially supported by CNPq (Proc.~303275/2013-8) and FAPERJ (Proc.~201.598/2014), and the fifth author was supported by a Trinity Hall Research Studentship.}

\begin{abstract}
Since their introduction by Erd\H{o}s in 1950, covering systems (that is, finite collections of arithmetic progressions that cover the integers) have been extensively studied, and numerous questions and conjectures have been posed regarding the existence of covering systems with various properties. In particular, Erd\H{o}s asked if the moduli can be distinct and all arbitrarily large, Erd\H{o}s and Selfridge asked if the moduli can be distinct and all odd, and Schinzel conjectured that in any covering system there exists a pair of moduli, one of which divides the other.

Another beautiful conjecture, proposed by Erd\H{o}s and Graham in 1980, states that if the moduli are distinct elements of the interval $[n,Cn]$, and $n$ is sufficiently large, then the density of integers uncovered by the union is bounded below by a constant (depending only on~$C$). This conjecture was confirmed (in a strong form) by Filaseta, Ford, Konyagin, Pomerance and Yu in 2007, who moreover asked whether the same conclusion holds if the moduli are distinct and sufficiently large, and $\sum_{i=1}^k \frac{1}{d_i} < C$. Although, as we shall see, this condition is not sufficiently strong to imply the desired conclusion, as one of the main results of this paper we will give an essentially best possible condition which is sufficient. More precisely, we show that if all of the moduli are sufficiently large, then the union misses a set of density at least $e^{-4C}/2$, where
\[
 C = \sum_{i=1}^k \frac{\mu(d_i)}{d_i}
\]
and $\mu$ is a multiplicative function defined by $\mu(p^i)=1+(\log p)^{3+\eps}/p$
for some $\eps > 0$. We also show that no such lower bound (i.e., depending only on~$C$) on the density of the uncovered set holds when $\mu(p^i)$ is replaced by any function of the form $1+O(1/p)$.

Our method has a number of further applications. Most importantly, as our second main theorem, we prove the conjecture of Schinzel stated above, which was made in 1967. We moreover give an alternative (somewhat simpler) proof of a breakthrough result of Hough, who resolved Erd\H{o}s' minimum modulus problem, with an improved bound on the smallest difference. Finally, we make further progress on the problem of Erd\H{o}s and Selfridge. 
\end{abstract}

\maketitle

\section{Introduction}

A \emph{covering system\/} is a finite collection $A_1,\dots,A_k$ of arithmetic progressions
that cover the integers, i.e., that satisfy $\bigcup_{i=1}^k A_i = \Z$. The study of covering
systems with distinct differences (or \emph{moduli}) $d_1 < \dots < d_k$ was initiated in 1950
by Erd\H{o}s~\cite{E50}, who used them to answer a question of Romanoff, and posed a number of problems regarding their properties. For example, Erd\H{o}s~\cite{E50} asked whether there exist such systems with minimum modulus arbitrarily large, Erd\H{o}s and Selfridge (see, e.g.,~\cite{FFK}) asked if there exists a covering system with all moduli distinct and odd, and Schinzel~\cite{Sch} conjectured that in any covering system there exist a pair of moduli, one of which divides the other. In 1980, Erd\H{o}s and Graham~\cite{EG} initiated the study of the density of the uncovered set; in particular, they conjectured that if the (distinct) moduli $d_1,\dots,d_k$ all lie in the interval $[n,Cn]$, where $n \ge n_0(C)$ is sufficiently large, then the uncovered set has density at least $\eps$ for some $\eps = \eps(C) > 0$.

The first significant progress on these problems was made by Filaseta, Ford, Konyagin, Pomerance and Yu~\cite{FFKPY} in 2007, who proved (in a strong form) the conjecture of Erd\H{o}s and Graham, and took an important step towards solving Erd\H{o}s' minimum modulus problem by showing that the sum of reciprocals of the moduli of a covering system with distinct differences grows (quickly) with the minimum modulus. Building on their work, and in a remarkable breakthrough, Hough~\cite{H} resolved the minimum modulus problem in 2015, showing that in every covering system with distinct moduli, the minimum modulus is at most $10^{16}$. The method of~\cite{H} was further refined by Hough and Nielsen~\cite{HN}, who used it to prove that every covering system with distinct differences contains a difference that is divisible by either 2 or 3. However, Hough's method does not appear to be strong enough to resolve the problem of Erd\H{o}s and Selfridge, and it moreover gives little information about the density of the uncovered set. 

The main aim of this paper is to develop a general method for bounding the density of the uncovered set. Our method, which is based on that of Hough, but is actually somewhat simpler, turns out to be sufficiently powerful and flexible to allow us to also make further progress on the problem of Erd\H{o}s and Selfridge, and to prove Schinzel's conjecture. Our starting point is the following natural and beautiful question of Filaseta, Ford, Konyagin, Pomerance and Yu~\cite{FFKPY}. 

\begin{qu*}
Is it true that for each $C > 0$, there exist constants $M > 0$ and $\eps > 0$ such that the following holds: for every covering system whose distinct moduli satisfy
\begin{equation}\label{eq:question}
d_1,\ldots,d_k \ge M \qquad \text{and} \qquad \sum_{i=1}^k \frac{1}{d_i} < C,
\end{equation}
the uncovered set has density at least $\eps$? 
\end{qu*}

In Section~\ref{sec:construc}, below, we will answer this question negatively for every $C \ge 1$, by constructing (a sequence of) families of arithmetic progressions with arbitrarily large moduli, for which the density of the uncovered set is arbitrarily small, and $\sum_{i=1}^k \frac{1}{d_i} < 1$. However, this immediately suggests the following question: what condition on the (distinct) moduli $d_1,\ldots,d_k$, in place of~\eqref{eq:question}, would allow us to deduce a lower bound (depending only on $C$) on the density of the uncovered set? Our main theorem provides a sufficient condition that is close to best possible.

\begin{theorem}\label{thm:uncoveredDensity}
 Let\/ $\eps > 0$ and let\/ $\mu$ be the multiplicative function defined by\/
 \begin{equation}\label{def:mu}
  \mu(p^i) = 1 + \frac{(\log p)^{3+\eps}}{p}
 \end{equation}
 for all primes\/ $p$ and integers\/~$i \ge1$. There exists\/ $M > 0$ so that if\/
 $A_1,\dots,A_k$ are arithmetic progressions with distinct moduli\/
 $d_1,\dots,d_k \ge M$, and
 \[
  C = \sum_{i=1}^k \frac{\mu(d_i)}{d_i},
 \]
 then the density of the uncovered set\/
 $R := \Z \setminus\bigcup_{i=1}^k A_i$ is at least\/~$e^{-4C}/2$.
\end{theorem}

Note that Hough's theorem is an immediate consequence of Theorem~\ref{thm:uncoveredDensity}. Our proof of Theorem~\ref{thm:uncoveredDensity} was inspired by that of Hough~\cite{H},
but is simpler in various important ways (for example, we do not need to appeal to the Lov\'asz
Local Lemma, and do not need his notion of quasi-randomness), and as a result we obtain a
somewhat simpler proof of his theorem, with a better bound on the minimum difference
(less than $10^6$). Our method of sieving, which (as we shall see) has a number of further applications, is outlined in Section~\ref{sec:sieve}.

We remark that the question of Filaseta, Ford, Konyagin, Pomerance and Yu~\cite{FFKPY} corresponds to replacing $\mu$ by the constant function~1. As noted above, we will show that the conclusion of the theorem does not hold under this weaker condition; in fact, we will prove that the theorem is close to best possible in the following much stronger sense. We will show (see Section~\ref{sec:construc}) that if \eqref{def:mu} is replaced by
\[
 \mu(p^i) = 1 + \frac{\lambda}{p}
\]
for any fixed $\lambda > 0$, then there exists a constant $C = C(\lambda) > 0$ such that the following holds: for every $M > 0$ and $\eps>0$, there exists a finite collection of arithmetic progressions, with distinct 
moduli $d_1,\dots,d_k \ge M$ satisfying $\sum_{i=1}^k\frac{\mu(d_i)}{d_i} \le C$, such that the uncovered set has density less than~$\eps$. It would be extremely interesting to characterize the functions $\mu$ such that, under the conditions of Theorem~\ref{thm:uncoveredDensity}, the density of the uncovered set is bounded from below by a constant $\delta(C) > 0$ depending only on $C$. 

Although our sieve was developed to control the density of the uncovered set, it turns out that it can be used to prove a number of additional interesting results about covering systems. We will focus here on the two classical examples mentioned above: the question of Erd\H{o}s and Selfridge, and the conjecture of Schinzel. Over 50 years ago, Erd\H{o}s and Selfridge (see~\cite{FFK} or~\cite{Sch}) asked whether or not there exist covering systems with distinct odd moduli.\footnote{Moreover, as recounted in~\cite{FFK}, Erd\H{o}s (who thought that such coverings are likely to exist) offered \$25 for a proof that there is no covering with these properties, and Selfridge (who expected the opposite) offered \$300 (later increased to \$2000) for a construction of such a covering.} Schinzel~\cite{Sch} showed that if no such covering system exists, then for every polynomial $f(x) \in \Z[X]$ with $f \not\equiv 1$, $f(0) \ne 0$ and $f(1) \ne -1$, there exists an (infinite) arithmetic progression of values of $n \in \Z$ such that $x^n + f(x)$ is irreducible over the rationals. He also showed that this would imply the following statement: in any covering system, one of the moduli divides another. In Section~\ref{Schinzel:sec} we will prove this latter statement, known as Schinzel's conjecture.


\begin{theorem}\label{thm:Schinzel}
If\/ $\cA$ is a finite collection of arithmetic progressions that covers the integers, then at least one of the moduli divides another.
\end{theorem}

Unfortunately, our method does not seem to be strong enough to resolve the Erd\H{o}s--Selfridge problem (see the discussion in Section~\ref{sec:general:method}). However, it does allow us to make some further progress towards a solution; in particular, we can prove that no such covering system exists under the additional constraint that the moduli are square-free. Since this application of our sieve requires several additional (somewhat technical) ideas, we will give the details elsewhere~\cite{SFES}. 

\begin{conj}\label{thm:SFES}
If\/ $\cA$ is a finite collection of arithmetic progressions with distinct square-free moduli that covers the integers, then at least one of the moduli is even.
\end{conj}

A different strengthening of the condition in the Erd\H{o}s--Selfridge problem was considered recently by Hough and Nielsen~\cite{HN}, who showed that in any covering system with distinct moduli, one of the moduli is divisible by either $2$ or $3$. Their proof required careful optimization of their techniques, and it seems difficult to use it to strengthen their result. Using our methods, we will give a short proof of the following strenthening of their theorem.

\begin{theorem}\label{thm:235}
 Let\/ $\cA=\{A_d : d \in D\}$ be a finite collection of arithmetic progressions with distinct
 moduli that covers the integers, and let $Q = \lcm(D)$ be the least common multiple of the
 moduli. Then either\/ $2\mid Q$, or\/ $9\mid Q$, or\/ $15\mid Q$.
\end{theorem}

In other words, either there is an even~$d$, a $d$ divisible by $3^2=9$, or there are $d_1,d_2\in D$ (possibly equal) with $3\mid d_1$ and $5\mid d_2$. We remark that we are unable to prove that a \emph{single\/} $d\in D$ has $15\mid d$ in this last case.

The rest of this paper is organized as follows: in Section~\ref{sec:sieve} we outline the sieve
we will use in the proofs, and in Section~\ref{sec:general} we state and prove our main technical
results, Theorems~\ref{thm:general} and~\ref{prop:moments}. In Section~\ref{sec:proof} we complete the proof of Theorem~\ref{thm:uncoveredDensity}, and in Section~\ref{sec:EGconj} we prove a variant of the main result of~\cite{FFKPY}. In Section~\ref{sec:general:method} we explain how our sieve can be used to prove the non-existence of coverings sets with certain properties, and in Sections~\ref{sec:ES}--\ref{Schinzel:sec} we use this method to improve Hough's bound on the minimum modulus, and to prove Theorems~\ref{thm:Schinzel} and~\ref{thm:235}. Finally, in Section~\ref{sec:construc}, we provide the constructions described above.

\section{The Sieve}\label{sec:sieve}

In this section we will outline the proof of Theorem~\ref{thm:uncoveredDensity}. We consider a
finite set of moduli denoted by $D\subseteq\N$ and a finite collection $\cA=\{A_d : d\in D\}$ of
arithmetic progressions, where $A_d = a_d + d\Z$ is an arithmetic progression with modulus~$d$.
The goal is to estimate the density of the uncovered set
\[
 R := \Z \setminus \bigcup_{d\in D} A_d.
\]

Rather than considering the entire collection of progressions $\cA$ all at once, we expose the
progressions `prime by prime' and track how the density of the uncovered set evolves.
To be more precise, let $p_1$, $p_2$, \dots, $p_n$ be the distinct prime factors of $Q := \lcm(D)$
(usually, but not necessarily, listed in increasing order) so that
\[
 Q = \prod_{j=1}^n p_j^{\gamma_j}
\]
for some integers $\gamma_j \ge 1$. Define, for each $0 \le i \le n$,
\[
 Q_i:=\prod_{j=1}^i p_j^{\gamma_j}
\]
and write
\[
 D_i := \big\{d\in D: d\mid Q_i\big\}\qquad \text{and} \qquad \cA_i:=\big\{A_d:d \in D_i\big\}
\]
for the set of differences and the corresponding arithmetic progressions whose prime factorization
only includes the first $i$ of these primes. (In particular, $Q_0=1$ and $D_0=\cA_0 = \emptyset$.)
Note that, although $\lcm(D_i)\mid Q_i$, we do not necessarily have $\lcm(D_i)=Q_i$, since earlier
primes can occur to higher powers in later moduli. Let
\[
 R_i:=\Z\setminus\bigcup_{d\in D_i}A_d = \Z\setminus\bigcup_{A_d\in\cA_i}A_d,
\]
be the set of elements not contained in any of the progressions of~$\cA_i$, so that $R_0=\Z$ and
$R_n = R$. We also write $N_i := D_i\setminus D_{i-1}$ for the set of ``new'' differences at
the $i$th stage, and define
\begin{equation}\label{def:B_i}
 B_i := \bigcup_{d\in N_i} A_d
\end{equation}
to be the union of the arithmetic progressions exposed at step~$i$, so that $R_i = R_{i-1} \setminus B_i$.

It will be convenient to consider $R_i$ as a subset of the cyclic group $\Z_{Q_i}$ (or of $\Z_Q$),
which is possible because for each $d \in D_i$ the set $A_d$ is periodic with period $d \mid Q_i$.
In particular, note that the density of~$R_i$ in $\Z$ is equal to the measure of the set $R_i$ in
the uniform probability measure on the finite set $\Z_{Q_i}$. During the proof we will in fact need
to consider \emph{non-uniform\/} probability measures $\Pr_i$ on $\Z_{Q_i}$; note that each such
measure can be extended (uniformly on each congruence class mod~$Q_i$) to a probability measure
on $\Z_Q$.\footnote{To be precise, we can set
$Q \cdot \Pr_i(x+Q\Z):=Q_i \cdot \Pr_i(x+Q_i\Z)$. Note that, since
$\gcd(Q_i,Q/Q_i)=1$, we can (via the Chinese Remainder Theorem) consider $\Pr_i$ on
$\Z_Q \cong \Z_{Q_i} \times \Z_{Q/Q_i}$ as a product measure of $\Pr_i$ on $\Z_{Q_i}$
with the uniform measure on $\Z_{Q/Q_i}$.} We will borrow (and abuse) terminology from measure
theory by calling a subset $S \subseteq \Z_Q$ (or a $Q$-periodic set $S \subseteq \Z$)
$Q_i$-\emph{measurable\/} if $S$ is a union of congruence classes mod~$Q_i$.

\subsection{A sketch of the method}

The basic idea is quite simple. We construct measures $\Pr_i$ in such a way that $\Pr_i(B_i)$
is small, but without changing the measure of $B_j$ for any $j < i$. It follows that the measure of $\Z\setminus R$ in the final measure $\Pr_n$ is at most $\sum_i \Pr_i(B_i)$, and thus if this quantity is less than 1, it follows that the arithmetic progressions do not cover~$\Z$. 

To bound $\Pr_i(B_i)$, we use the 1st and 2nd moment methods (see Lemma~\ref{lem:B_i}, below). More precisely, we bound the expectation (in the measure $\Pr_{i-1}$) of the proportion of the `fibre' $F(x) = \big\{ (x,y) : y \in \Z_{p_i^{\gamma_i}} \big\}$ of $x \in \Z_{Q_{i-1}}$ removed in step $i$, and the expectation of the square of this quantity. Bounding these moments is not too difficult, see Lemmas~\ref{lem:moments} and~\ref{lem:sums}, below. 

Finally, to deduce a lower bound on the uncovered set in the \emph{uniform\/} measure, $\Pr_0$, we shall need to bound the average distortion $\Pr_n(x)/\Pr_0(x)$. We will design the measures $\Pr_i$ so that the `average' here (which we take in the $\Pr_n$-measure) is enough (by a convexity argument) to give such a lower bound, see Lemma~\ref{lem:distortion}.


\subsection{The probability measures $\Pr_i$}

We will next define the non-uniform probability measures $\Pr_i$, which are inspired by (but different
in important ways to) a sequence of measures used in~\cite{H}, and which will play a key role in the
proof of Theorem~\ref{thm:uncoveredDensity}. The rough idea is that we would like to distort the space
so as to `blow up' the uncovered set, but without increasing the measure of any single point too much.
More precisely, we will define the measures inductively, choosing $\Pr_i$ so that it agrees with $\Pr_{i-1}$
on $Q_{i-1}$-measurable sets, is not too much larger that $\Pr_{i-1}$ anywhere, and (subject to these conditons) is as small as possible on the set~$B_i$, the union of the arithmetic progressions removed at step~$i$.

First, let $\Pr_0$ be the trivial probability measure on $\Z_{Q_0} = \Z_1$, or (equivalently, by the
comments above) the uniform measure on $\Z_Q$. Now fix $\delta_1,\dots,\delta_n \in [0,1/2]$,
let $i \in [n]$, and suppose we have already defined a probability measure $\Pr_{i-1}$ on $\Z_{Q_{i-1}}$.
Our task is to define a probability measure $\Pr_i$ on $\Z_{Q_i}$.

In order to view $R_{i-1}$ and $R_i$ as subsets of the same set, let us (by the Chinese Remainder Theorem)
represent $\Z_{Q_i}$ as $\Z_{Q_{i-1}} \times \Z_{p_i^{\gamma_i}}$, and denote elements of $\Z_{Q_i}$
by pairs $(x,y)$, where $x \in \Z_{Q_{i-1}}$ and $y \in \Z_{p_i^{\gamma_i}}$. We may view $R_{i-1}$
as a collection of `fibres' of the form $F(x) = \{ (x,y) : y \in \Z_{p_i^{\gamma_i}}\}$, noting that
$\Pr_{i-1}$ is uniform on each fibre, and view $R_i$ as being obtained from $R_{i-1}$ by removing
the points that are contained in the new progressions of $\cA_i \setminus\cA_{i-1}$.

Now, for each $x \in \Z_{Q_{i-1}}$, define
\begin{equation}\label{def:alpha}
 \alpha_i(x) = \frac{\Pr_{i-1}\big( F(x) \cap B_i \big)}{\Pr_{i-1}(x)}
 = \frac{\big| \big\{ y \in \Z_{p_i^{\gamma_i}} : (x,y)\in B_i \big\} \big|}{p_i^{\gamma_i}},
\end{equation}
that is, the proportion of the fibre $F(x)$ that is removed at stage~$i$.
The probability measure~$\Pr_i$ on $\Z_{Q_i}$ is defined as follows:
\begin{equation}\label{def:P_i}
 \Pr_i(x,y) :=
 \begin{cases}
   \max\bigg\{ 0, \, \ds\frac{\alpha_i(x)-\delta_i}{\alpha_i(x)(1-\delta_i)} \bigg\} \cdot \Pr_{i-1}(x,y),
   &\text{if }(x,y)\in B_i;\\[2ex]
   \min\bigg\{ \ds\frac{1}{1-\alpha_i(x)}, \, \frac{1}{1-\delta_i} \bigg\} \cdot \Pr_{i-1}(x,y),
   &\text{if }(x,y)\notin B_i.
 \end{cases}
\end{equation}
This is an important (and slightly technical) definition, and therefore deserves some additional
explanation. First, observe that if $\alpha_i(x) \le \delta_i$, then $\Pr_i(x,y) = 0$ for every
element of $Q_i$ that is covered in step~$i$, and that the measure is increased proportionally
elsewhere to compensate. On the other hand, for those $x \in \Z_{Q_{i-1}}$ for which
$\alpha_i(x) > \delta_i$, we `cap' the distortion by increasing the measure at each point not
covered in step $i$ by a factor of $1 / (1-\delta_i)$, and decreasing the measure on removed
points by a corresponding factor.

The following simple properties of the measure $\Pr_i$ will be useful in the proof.

\begin{lemma}\label{obs:Qmeas}
 $\Pr_i(S) = \Pr_{i-1}(S)$ for any\/ $Q_{i-1}$-measurable set\/~$S$.
\end{lemma}

\begin{proof}
Let $x \in \Z_{Q_{i-1}}$, and observe first that if $\alpha_i(x) \le \delta_i$, then
\[
 \Pr_i(x)=
  \Big(\alpha_i(x)\cdot 0 + \big(1-\alpha_i(x)\big) \cdot \frac{1}{1-\alpha_i(x)}\Big) \cdot \Pr_{i-1}(x)
  = \Pr_{i-1}(x),
\]
where $\Pr_i(x) = \sum_{(x,y) \in F(x)} \Pr_i(x,y)$. On the other hand, if $\alpha_i(x) > \delta_i$ then
\[
 \Pr_i(x)=
 \Big(\alpha_i(x)\cdot \frac{\alpha_i(x)-\delta_i}{\alpha_i(x)(1-\delta_i)}
  + \big(1-\alpha_i(x)\big) \cdot \frac{1}{1-\delta_i}\Big) \cdot \Pr_{i-1}(x) = \Pr_{i-1}(x).
\]
Summing over $x \in S$, we obtain $\Pr_i(S) = \Pr_{i-1}(S)$, as claimed.
\end{proof}

\begin{lemma}\label{obs:probs}
 For any set\/ $S \subseteq \Z_Q$, we have
\begin{equation}\label{e:dom}
 \Pr_i(S) \le \frac{1}{1-\delta_i} \cdot \Pr_{i-1}(S).
\end{equation}
Moreover, if $S \subseteq B_i$ then
\begin{equation}\label{e:baddom}
 \Pr_i(S) \le \Pr_{i-1}(S).
\end{equation}
\end{lemma}

\begin{proof}
Both inequalities follow immediately (for each `atom' $S=\{(x,y)\}$, and hence also in general,
by additivity) from~\eqref{def:P_i}.
\end{proof}

Given a function $f \colon \Z_Q \to\R_{\ge0}$, let us define the expectation of $f$ with respect
to $\Pr_i$ to be
\[
 \E_i\big[ f(x) \big] := \sum_{x \in \Z_Q} f(x) \Pr_i(x).
\]
By the observations above, we have
\[
 \E_i\big[ f(x) \big] \le \frac{1}{1-\delta_i} \cdot \E_{i-1}\big[f(x)\big],
\]
and moreover $\E_i[f(x)]\le\E_{i-1}[f(x)]$ if $f$ is supported on~$B_i$,
and $\E_i[f(x)]=\E_{i-1}[f(x)]$ if $f$ is $Q_{i-1}$-measurable.

\section{A general theorem}\label{sec:general}

In this section we will prove our main technical results, Theorems~\ref{thm:general} and~\ref{prop:moments}, below, which together imply Theorems~\ref{thm:uncoveredDensity},~\ref{thm:Schinzel} and~\ref{thm:235}, and also Hough's theorem.
In each case the deduction involves little more than choosing a suitable sequence $(\delta_1,\dots,\delta_n)$.

Given a finite collection $\cA = \{A_d : d\in D\}$ of arithmetic progressions, let $n$ be the number
of distinct prime factors of $Q = \lcm(D)$, and for each sequence $\delta_1,\dots,\delta_n \in [0,1/2]$,
let the probability distributions $\Pr_i$ and functions $\alpha_i \colon \Z_{Q_{i-1}} \to [0,1]$
be as defined in \eqref{def:alpha} and~\eqref{def:P_i}. Set
\[
 M_i^{(1)} := \E_{i-1}\big[\alpha_i(x)\big] \qquad \text{and} \qquad
 M_i^{(2)} := \E_{i-1}\big[\alpha_i(x)^2\big],
\]
and define a multiplicative function $\nu$, defined on factors of $Q$, by setting
\begin{equation}\label{def:nu}
 \nu(d) = \prod_{p_j\mid d} \frac{1}{1-\delta_j}.
\end{equation}
for each $d \mid Q$.

\begin{theorem}\label{thm:general}
 Let\/ $\cA = \{A_d : d\in D\}$ be a finite collection of arithmetic progressions, and let\/
 $\delta_1,\dots,\delta_n \in [0,1/2]$. If
 \begin{equation}\label{eq:eta}
  \eta := \sum_{i=1}^n \min\bigg\{ M_i^{(1)}, \frac{M_i^{(2)}}{4\delta_i(1-\delta_i)} \bigg\}  <  1,
 \end{equation}
 then\/ $\cA$ does not cover the integers. Moreover, the uncovered set\/ $R$ has density at least
 \begin{equation}\label{eq:density}
  \Pr_0(R) \ge \big(1-\eta\big) \exp\bigg(-\frac{2}{1-\eta}\sum_{d\in D}\frac{\nu(d)}{d}\bigg).
 \end{equation}
\end{theorem}

In order to show that~\eqref{eq:eta} holds in our applications, we need to bound the moments of~$\alpha_i(x)$. The following technical theorem provides general bounds that are sufficient in most cases.

\begin{theorem}\label{prop:moments}
 Let\/ $\cA = \{A_d : d\in D\}$ be a finite collection of arithmetic progressions, and let\/
 $\delta_1,\dots,\delta_n \in [0,1/2]$. Then
 \[
  M_i^{(1)} \le \sum_{mp_i^j\in N_i,\,m\mid Q_{i-1}} p_i^{-j} \cdot \frac{\nu(m)}{m}
  \le \frac{1}{p_i-1} \prod_{j < i} \bigg( 1 + \frac{1}{(1-\delta_j)(p_j-1)} \bigg),
 \]
 and
 \[
  M_i^{(2)} \le \sum_{\substack{m_1p_i^{j_1}, \, m_2p_i^{j_2}\in N_i\\ m_1,m_2 \,\mid\, Q_{i-1}}}
  p_i^{-(j_1+j_2)} \cdot \frac{\nu\big( \lcm(m_1,m_2) \big)}{\lcm(m_1,m_2)}
  \le \frac{1}{(p_i-1)^2}\prod_{j < i} \bigg( 1 + \frac{3p_j-1}{(1-\delta_j)(p_j-1)^2} \bigg).
 \]
\end{theorem}

The proofs of Theorems~\ref{thm:general} and~\ref{prop:moments} are both surprisingly simple. Let us begin with the following easy lemma, which is the first step in the proof of Theorem~\ref{thm:general}. We assume throughout this section that $\cA = \{A_d : d\in D\}$ is a given finite collection of arithmetic progressions such that $Q = \lcm(D)$ has exactly $n$ distinct prime factors, and fix a sequence $\delta_1,\dots,\delta_n \in [0,1/2]$, and hence a function~$\alpha_i$ and measure~$\Pr_i$ for each $i \in [n]$.

\begin{lemma}\label{lem:B_i}
 \[
  \Pr_i(B_i)\le\min\bigg\{\E_{i-1}\big[\alpha_i(x)\big],\
  \frac{\E_{i-1}\big[\alpha_i(x)^2 \big]}{4\delta_i(1-\delta_i)}\bigg\}.
 \]
\end{lemma}
\begin{proof}
Observe first that
\[
 \Pr_i(B_i) \le \Pr_{i-1}(B_i) = \E_{i-1}\big[\alpha_i(x)\big],
\]
where the inequality holds by~\eqref{e:baddom}, and the equality by~\eqref{def:alpha}.

For the other upper bound, we will use the elementary inequality $\max\{a-d,0\} \le a^2 / 4d$,
which is easily seen to hold for all $a,d > 0$ by rearranging the inequality $(a-2d)^2 \ge 0$.
By \eqref{def:alpha} and~\eqref{def:P_i} (the definitions of $\alpha_i$ and~$\Pr_i$), we have
\begin{align*}
 \Pr_i(B_i)&=\sum_{x \in \Z_{Q_{i-1}}}
  \max\bigg\{ 0, \, \frac{\alpha_i(x) - \delta_i}{\alpha_i(x)( 1 - \delta_i )} \bigg\} \cdot \Pr_{i-1}\big(F(x) \cap B_i \big)\\
 &=\frac{1}{1-\delta_i} \sum_{x \in \Z_{Q_{i-1}}} \max\big\{ 0, \, \alpha_i(x)-\delta_i \big\} \cdot \Pr_{i-1}(x) \\
 &\le \frac{1}{1-\delta_i}\sum_{x \in \Z_{Q_{i-1}}} \frac{\alpha_i(x)^2}{4\delta_i} \cdot \Pr_{i-1}(x)
  = \frac{\E_{i-1}\big[ \alpha_i(x)^2 \big]}{4\delta_i(1-\delta_i)},
\end{align*}
as required.
\end{proof}

It is already straightforward to deduce from Lemma~\ref{lem:B_i} and inequality~\eqref{eq:eta} that $\cA$
does not cover the integers. In order to deduce the bound~\eqref{eq:density} on the density of the uncovered set,
we will need to work slightly harder. First, we need the following easy bound on the $\Pr_i$-measure of an
arithmetic progression.

\begin{lemma}\label{lem:AP}
 For each\/ $0 \le i \le n$, and all\/ $b,d\in\Z$ such that\/ $d\mid Q$, we have
 \begin{equation}\label{e:APdist}
 \Pr_i\big(b+d\Z\big) \le \frac{1}{d} \prod_{p_j\mid d,\,j\le i}\frac{1}{1-\delta_j}
 = \frac{\nu\big(\gcd(d,Q_i)\big)}{d}.
\end{equation}
\end{lemma}
\begin{proof}
The proof is by induction on~$i$. Note first that $\Pr_0$ is just the uniform measure,
so $\Pr_0(b+d\Z) = 1/d$. So let $i \in [n]$, and assume that the claimed bound holds
for~$\Pr_{i-1}$. Suppose first that $p_i \mid d$. Then, by~\eqref{e:dom} and the induction
hypothesis, we have
\[
 \Pr_i\big( b + d\Z \big) \le \frac{1}{1-\delta_i} \cdot \Pr_{i-1}\big( b + d\Z \big)
 \le \frac{1}{d} \cdot \frac{1}{1-\delta_i} \prod_{p_j \mid d, \, j < i} \frac{1}{1-\delta_j}
 =\frac{1}{d} \prod_{p_j \mid d, \, j \le i} \frac{1}{1-\delta_j},
\]
as required. On the other hand, if $p_i \nmid d$ then we may write $d = m\ell$, where
$m = \gcd(d,Q_i) = \gcd(d,Q_{i-1})$. It follows that
\[
 \Pr_i\big(b + d\Z \big) = \frac{\Pr_i\big( b + m\Z \big)}{\ell}
 = \frac{\Pr_{i-1}\big( b + m\Z \big)}{\ell}
 \le \frac{1}{\ell m} \prod_{p_j \mid d, \, j < i} \frac{1}{1-\delta_j}
 = \frac{1}{d} \prod_{p_j \mid d, \, j \le i} \frac{1}{1-\delta_j},
\]
as required.
\end{proof}

Let us define the \emph{distortion\/} $\Delta_i(x)$ of a point $x \in \Z_{Q_i}$ by
\begin{equation}\label{def:distortion}
 \Delta_i(x) := \max\bigg\{ 0, \, \log \frac{\Pr_i(x)}{\Pr_0(x)} \bigg\}.
\end{equation}
The following bound on the average distortion will allow us to prove~\eqref{eq:density}.

\begin{lemma}\label{lem:distortion}
 For each\/ $0 \le i \le n$, we have
 \[
  \E_i\big[ \Delta_i(x) \big]  \le 2 \cdot \sum_{d \in D_i} \frac{\nu(d)}{d}.
 \]
\end{lemma}

\begin{proof}
We claim first that
\[
 \log \frac{\Pr_j(x)}{\Pr_{j-1}(x)} \le 2 \cdot \alpha_j(x)
\]
for every $j \in [n]$ and $x \in \Z_{Q_j}$. Indeed, observe that
$\Pr_{j-1}(x) / \Pr_j(x) \ge \max\{ 1 - \alpha_j(x),  1 - \delta_j \}$, by~\eqref{def:P_i},
and use the inequality $-\log(1-z) \le 2z$, which holds for $z \le 1/2$, and the fact that
$\delta_j \le 1/2$. It follows that
\[
 \E_i\big[ \Delta_i(x) \big]
 \le \sum_{j=1}^i \E_i \bigg[ \max\bigg\{ 0, \, \log \frac{\Pr_j(x)}{\Pr_{j-1}(x)} \bigg\} \bigg]
 \le 2 \cdot \sum_{j=1}^i \E_i\big[ \alpha_j(x)\big].
\]
Now, by~\eqref{def:alpha} and Lemma~\ref{obs:Qmeas}, we have
\[
 \E_i\big[ \alpha_j(x) \big] = \E_{j-1}\big[ \alpha_j(x) \big] = \Pr_{j-1}(B_j)
\]
for each $j \in [i]$, since the function $\alpha_j$ is $Q_{j-1}$-measurable.
Moreover, by~\eqref{def:B_i} (the definition  of $B_i$), the union bound, and Lemma~\ref{lem:AP},
we have
\[
 \Pr_{j-1}(B_j) \le \sum_{d \in N_j} \Pr_{j-1}(A_d)
 \le \sum_{d\in N_j}\frac{\nu(\gcd(d,Q_{j-1}))}{d}
 \le \sum_{d\in N_j}\frac{\nu(d)}{d}.
\]
Hence we obtain
\[
 \E_i\big[ \Delta_i(x) \big] \le 2 \cdot \sum_{j = 1}^i \sum_{d \in N_j} \frac{\nu(d)}{d}
 = 2 \cdot \sum_{d \in D_i} \frac{\nu(d)}{d},
\]
as claimed.
\end{proof}

Theorem~\ref{thm:general} now follows easily from Lemmas~\ref{lem:B_i} and~\ref{lem:distortion}.

\begin{proof}[Proof of Theorem~\ref{thm:general}]
We claim first that
\begin{equation}\label{gen:proof:first:step}
 1 - \Pr_n(R) \le \sum_{i=1}^n \Pr_n(B_i) = \sum_{i=1}^n \Pr_i(B_i) \le \eta.
\end{equation}
Indeed, the first inequality is just the union bound; the equality holds by Lemma~\ref{obs:Qmeas},
since $B_i$ is $Q_i$-measurable, so $\Pr_i(B_i) = \Pr_{i+1}(B_i) = \dots = \Pr_n(B_i)$;
and the final inequality follows from Lemma~\ref{lem:B_i} and~\eqref{eq:eta}, the definition of~$\eta$.
It follows that $\Pr_n(R) \ge 1 - \eta > 0$ if $\eta<1$, and hence $R$ is non-empty, i.e., $\cA$ does
not cover the integers.

To prove the claimed lower bound on the density of the uncovered set, we will use Lemma~\ref{lem:distortion}.
Indeed, by the definition~\eqref{def:distortion} of $\Delta_n(x)$, we have
\[
 \Pr_0(R) = \E_0\big[\id_{x\in R}\big] \ge \E_n\big[\id_{x \in R}\exp\big(-\Delta_n(x)\big)\big].
\]
Now, by the convexity of $e^{-z}$, and noting that
$\Pr_n(R) \cdot \E_n\big[ \Delta_n(x) \mid x \in R \big] \le \E_n\big[ \Delta_n(x) \big]$,
\begin{align*}
 \E_n\big[\id_{x \in R} \exp\big(-\Delta_n(x) \big) \big]
 &= \Pr_n(R) \cdot \E_n\big[\exp\big(-\Delta_n(x) \big) \bmid x \in R \big] \\[+1ex]
 &\ge \Pr_n(R) \cdot \exp\big(-\E_n\big[ \Delta_n(x) \mid x \in R \big] \big)\\
 &\ge \Pr_n(R) \cdot \exp\bigg(-\frac{\E_n\big[\Delta_n(x)\big]}{\Pr_n(R)} \bigg).
\end{align*}
Hence, by Lemma~\ref{lem:distortion}, and since $\Pr_n(R) \ge 1 - \eta$, by~\eqref{gen:proof:first:step}, we obtain
\[
 \Pr_0(R) \ge \big(1-\eta\big) \exp\bigg(-\frac{2}{1-\eta}\sum_{d \in D}\frac{\nu(d)}{d}\bigg),
\]
as required.
\end{proof}

\subsection{Bounding the moments of $\alpha_i(x)$}

The proof of Theorem~\ref{prop:moments} is also quite straightforward. First, recall
that $N_i := D_i\setminus D_{i-1}$ is the set of new differences at step~$i$, and note that
any $d \in N_i$ can be represented in the form $d = mp_i^j$, where $m \mid Q_{i-1}$ and
$1 \le j \le \gamma_i$. 
The first step is the following general bound on the moments of $\alpha_i(x)$.

\begin{lemma}\label{lem:moments}
For each\/ $k \in \N$,
\begin{align}
\E_{i-1}\big[ \alpha_i(x)^k \big] & \le \sum_{\substack{m_1p_i^{j_1},\dots,m_kp_i^{j_k} \in N_i \\ m_1,\dots,m_k \,\mid\, Q_{i-1}}} \frac{1}{p_i^{j_1+\dots+j_k}} \cdot \frac{\nu\big(\lcm(m_1,\dots,m_k)\big)}{\lcm(m_1,\dots,m_k)} \label{eq:moments:one}\\
& \le \frac{1}{(p_i - 1)^k} \sum_{m_1,\dots,m_k \,\mid\, Q_{i-1}} \frac{ \nu\big( \lcm(m_1,\dots,m_k) \big)}{\lcm(m_1,\dots,m_k)}.\label{eq:moments:two}
\end{align}
\end{lemma}

\begin{proof}
Recall the definitions~\eqref{def:B_i} and~\eqref{def:alpha} of the set $B_i$ and the
function~$\alpha_i$, respectively. Applying the union bound, we obtain, for each $x \in \Z_{Q_{i-1}}$,
\[
 \alpha_i(x) \, = \sum_{(x,y) \in \Z_{Q_i}} p_i^{-\gamma_i} \cdot \id\big[ (x,y) \in B_i \big]
 \le \sum_{(x,y) \in \Z_{Q_i}} \sum_{d \in N_i} p_i^{-\gamma_i} \cdot \id\big[ (x,y) \in A_d \big].
\]
Now, observe that if $d = mp_i^j \in N_i$, where $p_i \nmid m$, then there are either zero or
$p^{\gamma_i - j}$ values of $y \in \Z_{p_i^{\gamma_i}}$ with $(x,y) \in A_d$. Indeed,
$(x,y) \in A_d = a_d + d\Z$ iff $x \equiv a_d \bmod m$ and $y\equiv a_d \bmod p_i^j$.
It follows that
\[
 \alpha_i(x) \le \sum_{d=mp_i^j\in N_i} p_i^{-j} \cdot \id\big[ x \equiv a_d \bmod m \big],
\]
and hence
\[
 \E_{i-1}\big[ \alpha_i(x)^k \big] \le \sum_{d_1 = m_1p_i^{j_1} \in N_i}\cdots
 \sum_{d_k = m_kp_i^{j_k} \in N_i} \frac{1}{p_i^{j_1+\dots+j_k}}
 \Pr_{i-1}\big( x \equiv a_{d_j} \bmod {m_j}\text{ for }j\in[k]\big).
\]
Now, note that the intersection of the events $\{ x \equiv a_{d_j} \bmod m_j \}$ over $j \in [k]$
is either empty (if the congruences are incompatible), or is equivalent to
$x \equiv b \bmod\lcm(m_1,\dots,m_k)$ for some~$b$. Therefore, by Lemma~\ref{lem:AP}, we have
\[
 \Pr_{i-1}\big( x \equiv a_{d_j} \bmod {m_j}\text{ for }j\in[k] \big)
 \le \frac{ \nu\big( \lcm(m_1,\dots,m_k) \big)}{\lcm(m_1,\dots,m_k)},
\]
and hence
\begin{align*}
\E_{i-1}\big[ \alpha_i(x)^k \big] & \le \sum_{\substack{m_1p_i^{j_1},\dots,m_kp_i^{j_k} \in N_i \\ m_1,\dots,m_k \,\mid\, Q_{i-1}}} \frac{1}{p_i^{j_1+\dots+j_k}} \cdot \frac{ \nu\big( \lcm(m_1,\dots,m_k) \big)}{\lcm(m_1,\dots,m_k)}\\
& \le \sum_{j_1,\dots,j_k\ge 1} \frac{1}{p_i^{j_1+\dots+j_k}} \sum_{m_1,\dots, m_k \,\mid\, Q_{i-1}} \frac{ \nu\big( \lcm(m_1,\dots,m_k) \big)}{\lcm(m_1,\dots,m_k)}.
\end{align*}
This proves~\eqref{eq:moments:one}; to obtain~\eqref{eq:moments:two}, simply note that
\[
 \sum_{j_1,\dots,j_k\ge1}\frac{1}{p_i^{j_1+\dots+j_k}}
 = \bigg(\sum_{j\ge1}\frac{1}{p_i^j}\bigg)^k
 = \frac{1}{(p_i-1)^k}.\qedhere
\]
\end{proof}

To complete the proof of Theorem~\ref{prop:moments}, it only remains to prove the following bounds.

\begin{lemma}\label{lem:sums}
 For each\/ $i \in [n]$,
 \[
\sum_{m \mid Q_{i-1}} \frac{\nu(m)}{m} \le \, \prod_{j < i} \bigg( 1 + \frac{1}{(1-\delta_j)(p_j-1)} \bigg)
 \]
 and
 \[
\sum_{m_1,m_2 \mid Q_{i-1}} \frac{\nu\big( \lcm(m_1,m_2) \big)}{\lcm(m_1,m_2)} \le \, \prod_{j < i} \bigg( 1 + \frac{3p_j-1}{(1-\delta_j)(p_j-1)^2} \bigg).
 \]
\end{lemma}

\begin{proof}
Recall from~\eqref{def:nu} that $\nu$ is a multiplicative function. It follows that
\[
\sum_{m \mid Q_{i-1}} \frac{\nu(m)}{m} = \, \prod_{j < i} \sum_{t = 0}^{\gamma_j} \frac{\nu(p_j^t)}{p_j^t},
\]
and by~\eqref{def:nu}, we have
\[
 \prod_{j < i} \sum_{t = 0}^{\gamma_j} \frac{\nu(p_j^t)}{p_j^t}
 = \, \prod_{j < i} \bigg( 1 + \frac{1}{1-\delta_j} \sum_{t = 1}^{\gamma_j} \frac{1}{p_j^t} \bigg)
 \le \, \prod_{j < i} \bigg( 1 + \frac{1}{(1-\delta_j)(p_j-1)} \bigg).
\]

\pagebreak
To prove the second inequality, let us write $\chi(m)$ for the number of ways of representing
a number $m > 1$ as the least common multiple of two numbers, i.e.,
\[
 \chi(m) :=  \big| \big\{ (m_1,m_2) : \lcm(m_1,m_2) = m \big\}\big|,
\]
so that
\[
\sum_{m_1,m_2\mid Q_{i-1}} \frac{\nu\big( \lcm(m_1,m_2) \big)}{\lcm(m_1,m_2)} = \sum_{m \mid Q_{i-1}} \frac{\chi(m)\nu(m)}{m}.
\]
Observe that the function $\chi$ is multiplicative and satisfies $\chi(p^t) = 2t+1$ for all primes $p$
and integers $t \ge 0$. It follows that
\begin{equation}\label{eq:sum:of:powers}
 \sum_{m \mid Q_{i-1}} \frac{\chi(m)\nu(m)}{m}
 = \, \prod_{j < i} \sum_{t = 0}^{\gamma_i} \frac{\chi(p_j^t) \nu(p_j^t)}{p_j^t}
 \le \, \prod_{j < i} \bigg( 1 + \sum_{t  = 1}^\infty \frac{2t+1}{(1 - \delta_j)p_j^t} \bigg).
\end{equation}
Finally, note that for any $p > 1$, we have
\begin{equation}\label{eq:sum:odds:over:powers}
 \sum_{t \ge 1}\frac{2t+1}{p^t}
 = \frac{1}{(1-p^{-1})} \bigg( \frac{3}{p} + \frac{2}{p^2} + \frac{2}{p^3} + \dots \bigg)
 = \frac{1}{(1-p^{-1})^2} \bigg( \frac{3}{p} - \frac{1}{p^2} \bigg) = \frac{3p-1}{(p-1)^2}.
\end{equation}
This completes the proof of the lemma.
\end{proof}

Theorem~\ref{prop:moments} is an almost immediate consequence of Lemmas~\ref{lem:moments} and~\ref{lem:sums}.

\begin{proof}[Proof of Theorem~\ref{prop:moments}]
To bound $M_i^{(1)}$ we apply Lemma~\ref{lem:moments} with $k = 1$ and the first inequality in Lemma~\ref{lem:sums}. This gives
$$\E_{i-1}\big[ \alpha_i(x) \big] \le \sum_{\substack{m p_i^{j} \in N_i \\ m \mid Q_{i-1}}} \frac{1}{p_i^{j}} \cdot \frac{\nu(m)}{m} \le \frac{1}{p_i - 1} \sum_{m \mid Q_{i-1}}
\frac{ \nu(m)}{m} \le \frac{1}{p_i - 1} \prod_{j < i} \bigg( 1 + \frac{1}{(1-\delta_j)(p_j-1)} \bigg),$$
as claimed. To bound $M_i^{(2)}$ we apply Lemma~\ref{lem:moments} with $k = 2$ and the second inequality in Lemma~\ref{lem:sums}. We obtain
\begin{multline*}
\E_{i-1}\big[ \alpha_i(x)^2 \big] \le \sum_{\substack{m_1p_i^{j_1}, \, m_2p_i^{j_2}\in N_i\\ m_1,m_2 \mid Q_{i-1}}} p_i^{-(j_1+j_2)} \cdot \frac{\nu\big(\lcm(m_1,m_2)\big)}{\lcm(m_1,m_2)} \\
\le \frac{1}{(p_i - 1)^2} \sum_{m_1,m_2 \mid Q_{i-1}}
\frac{ \nu\big( \lcm(m_1,m_2) \big)}{\lcm(m_1,m_2)} \le \frac{1}{(p_i - 1)^2} \prod_{j < i} \bigg( 1 + \frac{3p_j-1}{(1-\delta_j)(p_j-1)^2} \bigg),
\end{multline*}
as required.
\end{proof}

\pagebreak
\section{Proof of the Main Theorem}\label{sec:proof}

In order to deduce Theorem~\ref{thm:uncoveredDensity} from Theorems~\ref{thm:general} and~\ref{prop:moments}, it will suffice to show that there is an appropriate choice
of $M$ and $\delta_1,\delta_2,\dots,\delta_n$.

\begin{proof}[Proof of Theorem~\ref{thm:uncoveredDensity}]
Let $p_1,\dots,p_n$ be the primes that divide~$Q$, listed in increasing order, and fix an integer~$k^*$,
to be determined later. We set $\delta_i = 0$ for $i \le k^*$ and
\[
 \delta_i = \frac{\mu(p_i)-1}{\mu(p_i)}
 = \frac{(\log p_i)^{3+\eps}}{p_i} \bigg(1 + \frac{(\log p_i)^{3+\eps}}{p_i} \bigg)^{-1}
\]
for $i > k^*$. Note that we have $\delta_1,\dots,\delta_n \in [0,1/2]$ if $k^*$ is chosen sufficiently large.

In order to apply Theorem~\ref{thm:general}, we will bound $M_i^{(1)}$ for each $i \le k^*$, and $M_i^{(2)}$ for $i > k^*$. We will do so using Theorem~\ref{prop:moments}.

\begin{claim}\label{claim1}
 For any choice of\/~$k^*$, if\/ $M$ is sufficiently large then
 \[
  \sum_{i\le k^*} M_i^{(1)} \le \frac{1}{4}.
 \]
\end{claim}
\begin{proof}[Proof of Claim~\ref{claim1}]
Note first that $\nu(d) = 1$ for every $d \in D_{k^*}$, by~\eqref{def:nu} and our choice
of $\delta_1,\dots,\delta_n$. Thus, by Theorem~\ref{prop:moments}, we have
\[
 \sum_{i \le k^*} M_i^{(1)}
 \le \sum_{i \le k^*} \sum_{mp_i^j\in N_i} p_i^{-j} \cdot \frac{\nu(m)}{m}
 = \sum_{i \le k^*} \sum_{d \in N_i} \frac{1}{d}
 = \sum_{d \in D_{k^*}} \frac{1}{d}.
\]
Now, let $\cS(q)$ denote the set of $q$-\emph{smooth\/} numbers, i.e., numbers all of whose prime
factors are at most $q$, and note that $D_{k^*} \subseteq \cS(p_{k^*})$, and moreover
\[
 \sum_{d \in \cS(q)} \frac{1}{d} \, = \, \prod_{p \le q} \bigg( 1 + \frac{1}{p - 1} \bigg) < \infty
\]
for every $q$, where the product is over primes $p \le q$. Hence, by choosing $M$ to be sufficiently large,
it follows that
\[
 \sum_{i \le k^*} M_i^{(1)} \le \sum_{d \in \cS(p_{k^*}), \, d \ge M} \frac{1}{d} \, \le \, \frac{1}{4},
\]
as claimed.
\end{proof}

Bounding $M_i^{(2)}$ for $i > k^*$ is only slightly less trivial.

\begin{claim}\label{claim2}
 If\/ $k^*$ is sufficiently large, then
 \[
  \sum_{i > k^*} M_i^{(2)} \le \frac{1}{4}.
 \]
\end{claim}
\begin{proof}[Proof of Claim~\ref{claim2}]
Recall from Theorem~\ref{prop:moments} that
\[
 M_i^{(2)} \le \frac{1}{(p_i-1)^2} \prod_{j < i} \bigg( 1 + \frac{3p_j-1}{(1-\delta_j)(p_j-1)^2} \bigg)
 \le \frac{1}{(p_i-1)^2} \exp\bigg( \sum_{j < i} \frac{3p_j-1}{(1-\delta_j)(p_j-1)^2} \bigg).
\]
Now, by our choice of $\delta_1,\dots,\delta_n$, we have
\[
 \frac{3p_j-1}{(1-\delta_j)(p_j-1)^2} \le \frac{3}{p_j} + \frac{O\big((\log p_j)^{3+\eps}\big)}{p_j^2}
\]
for every $j < i$, and hence, using a weak form of Merten's theorem to deduce that
$\sum_{j < i} \frac{1}{p_j} \le \log\log p_i + O(1)$, we obtain
\[
 M_i^{(2)} \le \frac{1}{(p_i-1)^2} \exp\big( 3\log\log p_i + O(1) \big) \le \frac{C_0(\log p_i)^3}{p_i^2}
\]
for some absolute constant $C_0 > 0$. It follows that, if $i > k^*$, then
\[
 \frac{M_i^{(2)}}{4\delta_i(1-\delta_i)} = \frac{\mu(p_i)^2}{4\big( \mu(p_i) - 1 \big)} \cdot M_i^{(2)}
 \le \frac{C_1}{p_i (\log p_i)^\eps},
\]
for some absolute constant $C_1 > 0$. Now, using the prime number theorem $\pi(x) \sim x / \log x$
to crudely bound the sum of the right-hand side over all primes~$p$, we obtain
\[
 \sum_{p\ge3} \frac{C_1}{p (\log p)^\eps} = \sum_{t \ge 1} \sum_{e^t < p \le e^{t+1}} \frac{C_1}{p(\log p)^\eps}
 \le \sum_{t \ge 1} \frac{C_1 \cdot \pi\big( e^{t+1} \big)}{e^t \cdot t^\eps}
 \le \sum_{t \ge 1} \frac{C_2}{t^{1+\eps}} < \infty
\]
for any $\eps > 0$. It follows that if $k^*$ is sufficiently large, then
\[
 \sum_{i > k^*} \frac{M_i^{(2)}}{4\delta_i(1-\delta_i)} \le \frac{1}{4}.
\]
as claimed.
\end{proof}

By Claims~\ref{claim1} and~\ref{claim2}, it follows that $\eta \le 1/2$, and hence, by
Theorem~\ref{thm:general}, the uncovered set has density at least
\[
 \frac{1}{2}\exp\bigg( - 4 \sum_{d \in D} \frac{\nu(d)}{d} \bigg).
\]
Finally, observe that $\mu(d) \ge \nu(d)$ for every $d \in \N$, since both functions are multiplicative,
\[
 \mu(p) = 1 + \frac{(\log p)^{3+\eps}}{p} = \frac{1}{1-\delta_i} = \nu(p)
\]
for $p > p_{k^*}$, and $\mu(p) \ge 1 = \nu(p)$ for $p \le p_{k^*}$.
It follows that the uncovered set has density at least
\[
 \frac{1}{2}\exp\bigg( - 4 \sum_{d \in D} \frac{\mu(d)}{d} \bigg) = \frac{e^{-4C}}{2},
\]
as required.
\end{proof}

\section{The Erd\H{o}s--Graham Conjecture}\label{sec:EGconj}

As a simple consequence of Theorem~\ref{thm:uncoveredDensity}, we will next give a new strengthening of the conjecture of Erd\H{o}s and Graham~\cite{EG} mentioned in the Introduction. Recall that in the original conjecture, which was confirmed by Filaseta, Ford, Konyagin, Pomerance and Yu~\cite{FFKPY} in 2007, the lower bound $M$ on $n$ was allowed to depend on $K$. The result proved in~\cite{FFKPY} required such a bound (of the form $n > K^\omega$ for some $\omega = \omega(K) \to \infty$ as $K \to \infty$), but gave an asymptotically optimal bound on the density of the uncovered set. Our result gives a non-optimal bound on this density, but does not require $n$ to grow with $K$.

\begin{theorem}
 There exists\/ $M$ such that for any\/ $K \ge 1$, there exists\/ $\delta = \delta(K) > 0$ such that the following holds.
 If\/ $A_1,\dots,A_k$ are arithmetic progressions with distinct moduli\/ $d_1,\dots,d_k \in [n,Kn]$, $n\ge M$,
 then the uncovered set\/ $R := \Z \setminus\bigcup_{i=1}^k A_i$ has density at least\/~$\delta$.
\end{theorem}

\begin{proof}
We apply Theorem~\ref{thm:uncoveredDensity} with $\eps = 1$. In order to prove the corollary, it will suffice to show
that there exists a constant $C > 0$, depending only on~$K$, such that
\begin{equation}\label{e:sbound}
 \sum_{d \in [n,Kn]} \frac{\mu(d)}{d} \le C
\end{equation}
for all $n \in \N$. Indeed, if we set $\delta := e^{-4C}/2$ then it follows from Theorem~\ref{thm:uncoveredDensity} that,
if \eqref{e:sbound} holds and $n \ge M$, then the density of $R$ is at least~$\delta$.

We will prove, by induction on $t \ge 0$, that \eqref{e:sbound} holds for all $n \le 2^t$. This is clearly the case when $t = 0$
(assuming that $C = C(K)$ is sufficiently large), so let $t \ge 1$, and let us assume that~\eqref{e:sbound} holds for all $n \le 2^{t-1}$.
We will use the following telescopic series, which holds for any $i_0$ and $d \mid Q$ by the definition~\eqref{def:mu} of $\mu$:
\begin{align*}
 \mu(d) &= \mu\big( \gcd(d,Q_{i_0}) \big) + \sum_{i > i_0} \Big( \mu\big( \gcd(d,Q_i) \big) - \mu\big( \gcd(d,Q_{i-1}) \big) \Big)\\
 &= \mu\big( \gcd(d,Q_{i_0}) \big) + \sum_{p_i \mid d, \, i > i_0} \big( \mu(p_i) - 1 \big)\mu\big( \gcd(d,Q_{i-1}) \big).
\end{align*}
Summing over $d \in [n,Kn]$, we obtain
\[
 \sum_{d \in [n,Kn]} \frac{\mu(d)}{d} \le \sum_{d \in [n,Kn]} \frac{\mu(Q_{i_0})}{d}
  + \sum_{i > i_0} \frac{\mu(p_i) - 1}{p_i} \sum_{d \in [n/p_i,Kn/p_i]}\frac{\mu(d)}{d}.
\]
Now, since $n/p \le 2^{t-1}$ for every prime~$p$, it follows by the induction hypothesis that
\[
 \sum_{d \in [n,Kn]} \frac{\mu(d)}{d} \le \mu(Q_{i_0}) \big( \log K + 1 \big) + C \sum_{i > i_0} \frac{\mu(p_i) - 1}{p_i}.
\]
Finally, note that the sum over all primes
\[
 \sum_{p \text{ prime}} \frac{\mu(p) - 1}{p} = \sum_{p \text{ prime}} \frac{(\log p)^4}{p^2}
\]
converges, and therefore
\[
 \sum_{d \in [n,Kn]} \frac{\mu(d)}{d} \le \mu(Q_{i_0})\big( \log K + 1 \big) + \frac{C}{2} \le C
\]
if $i_0$ and then $C = C(K)$ are chosen sufficiently large, as required.
\end{proof}

\section{A general method of applying the sieve}\label{sec:general:method}

In this section we describe a practical method of applying our method to problems involving covering systems, such as Schinzel's conjecture and the Erd\H{o}s--Selfridge problem. More precisely, we will show how one can choose  the constants $\delta_i$ sequentially and optimally via a simple recursion which may be run on a computer. We will also give a simple criterion (see Theorem~\ref{thm:f:sufficient}, below) which we prove is sufficient to deduce that the collection $\cA$ does not cover the integers. Combining these (that is, running the recursion until the criterion is satisfied), we reduce the problems to finite calculations, which in some cases are tractable. To demonstrate the power of this approach, we will use it in Sections~\ref{sec:ES}--\ref{Schinzel:sec} to prove Theorems~\ref{thm:Schinzel} and~\ref{thm:235}, and to significantly improve the bound on $M$ in Hough's theorem. 

We begin by defining a sequence of numbers $f_k = f_k(\cA)$, which will (roughly speaking) encode how ``well" we are doing after $k$ steps of our sieve. We remark that, from now on, we will perform a step of the sieve for \emph{every} prime, whether or not it divides $Q$. We will therefore write $p_k$ for the $k$th prime, i.e., $p_1 = 2$, $p_2 = 3$, etc. Fix $i_0 \in \N$, and define 
\begin{equation}\label{fk:def}
f_k = f_k(\cA) := \frac{\kappa}{\mu_k} \prod_{i_0 < i \le k}\bigg( 1+\frac{3p_i-1}{(1-\delta_i)(p_i-1)^2} \bigg)
\end{equation}
for each $k \ge i_0$, where 
$$\mu_k := 1 - \sum_{i \le k} \Pr_i(B_i),$$
and $\kappa > 0$ and $i_0 \in \N$ are chosen so that
\begin{equation}\label{kappa:def:assumption}
 M^{(2)}_k \le \frac{\kappa}{(p_k - 1)^2} \prod_{i_0 < i < k}\bigg( 1+\frac{3p_i-1}{(1-\delta_i)(p_i-1)^2} \bigg) = \frac{\mu_{k-1} f_{k-1}}{(p_k - 1)^2}
\end{equation}
for every $k > i_0$. For example, by Theorem~\ref{prop:moments},
\begin{equation}\label{kappa:sufficient}
\kappa = \prod_{i \le i_0, \, p_i \mid Q}\bigg( 1+\frac{3p_i-1}{(1-\delta_i)(p_i-1)^2} \bigg)
\end{equation}
is a valid choice, although in some cases we can prove a stronger bound. We remark that, in practice, we will choose the constants $\kappa$ and $\delta_1,\ldots,\delta_{i_0}$, and show that~\eqref{kappa:def:assumption} holds for any sequence $(\delta_{i_0+1},\ldots,\delta_n)$. We will then choose each subsequent $\delta_i$ so as to minimize $f_i$. 

The following theorem gives a sufficient condition (at step $k$) for our sieve to be successful. The bound we prove gives an almost optimal termination criterion when $k$ is large.

\pagebreak

\begin{theorem}\label{thm:f:sufficient}
Let\/ $k \ge 10$. If\/ $\mu_k > 0$ and\/ $f_k(\cA) \le (\log k + \log\log k - 3)^2k$, then the system of arithmetic progressions $\cA$ does not cover\/~$\Z$.
\end{theorem}

In this section, it will be convenient to define, for each $i \in \N$, 
\begin{equation}\label{def:ab}
 a_i=\frac{3p_i-1}{(p_i-1)^2}\qquad\text{and}\qquad
 b_i=\frac{1}{4(p_i-1)^2}.
\end{equation}
The first step in the proof of Theorem~\ref{thm:f:sufficient} is the following simple (but key) lemma. 

\begin{lemma}\label{lem:f:recursion}
Let $i > i_0$, and assume that $\mu_{i-1} > 0$. If $b_if_{i-1} < \delta_i(1-\delta_i)$, then $\mu_i > 0$, and
\begin{equation}\label{e:f_i}
 f_i \, \le \bigg( 1+\frac{a_i}{1-\delta_i} \bigg) \bigg( 1 - \frac{b_if_{i-1}}{\delta_i(1-\delta_i)} \bigg)^{-1} f_{i-1}.
\end{equation}
\end{lemma}

\begin{proof}
Recall from Lemma~\ref{lem:B_i} and~\eqref{kappa:def:assumption}, that 
\begin{equation}\label{eq:mu:step}
\mu_{i-1} - \mu_i \, = \, \Pr_i(B_i) \, \le \, \frac{M^{(2)}_i}{4\delta_i(1-\delta_i)} \, \le \, \frac{\mu_{i-1} f_{i-1}}{4\delta_i(1-\delta_i)(p_i - 1)^2}.
\end{equation}
It follows that $\mu_i \ge \mu_{i-1} \big( 1 - b_if_{i-1} / \delta_i(1-\delta_i) \big) > 0$, and moreover
\begin{align*}
\frac{f_i}{f_{i-1}} & \, = \, \frac{\mu_{i-1}}{\mu_i} \bigg( 1+\frac{3p_i-1}{(1-\delta_i)(p_i-1)^2} \bigg) \\
& \, \le \, \bigg( 1 + \frac{3p_i-1}{(1-\delta_i)(p_i-1)^2} \bigg) \bigg( 1 - \frac{f_{i-1}}{4\delta_i(1-\delta_i)(p_i-1)^2} \bigg)^{-1},
\end{align*}
as claimed.
 \end{proof}

To deduce Theorem~\ref{thm:f:sufficient}, we will use the main result (which is also the title) of~\cite{D}, which states that for each $k \ge 2$, the $k$th prime is greater than $k(\log k + \log\log k-1)$.

\begin{proof}[Proof of Theorem~\ref{thm:f:sufficient}]
We are required to show that $\mu_n > 0$; to do so, we will use Lemma~\ref{lem:f:recursion} to show, by induction on $i$, that $\mu_i > 0$ for every $k \le i \le n$. As part of the induction, we will also prove that $f_i \le \lambda_i^2 i$ for each $k \le i \le n$, where $\lambda_i = \log i + \log\log i - 3$. 

Note that the base case, $i = k$, follows from our assumptions, and set $\delta_i = 1/2$ for each $k < i \le n$. By Lemma~\ref{lem:f:recursion}, for the induction step it will suffice to show that  $4b_if_{i-1} < 1$ and
\[
\lambda_{i-1}^2 (i - 1) \bigg( \frac{1+2a_i}{1-4b_if_{i-1}} \bigg) \le \lambda_i^2 i,
\]
where $a_i$ and $b_i$ were defined in~\eqref{def:ab}. We claim that, writing $\lambda=\lambda_{i-1}$, we have
\[
2a_i \le \frac{6(\lambda+2)i + 4}{(\lambda+2)^2i^2},\qquad 4b_if_{i-1} \le \frac{\lambda^2(i-1)}{(\lambda+2)^2i^2} < 1,\qquad \text{and}\qquad \lambda_i^2i \ge \lambda (\lambda i+2).
\]
To see these inequalities, note that $p_i \ge (\lambda_i + 2) i$ for all $i \ge k$, by the result of~\cite{D} stated above, and $\lambda_i - \lambda_{i-1} \ge \log(i/(i-1)) > 1/i$, so $p_i - 1 \ge (\lambda_{i-1} + 1/i + 2)i - 1 = (\lambda + 2)i$.

It is therefore sufficient to show that
\[
\lambda^2(i-1) \bigg( 1 + \frac{6(\lambda+2)i + 4}{(\lambda+2)^2i^2} \bigg) \le \lambda(\lambda i+2) \bigg( 1 -  \frac{\lambda^2(i-1)}{(\lambda+2)^2i^2} \bigg),
\]
which (after multiplying by $(\lambda+2)^2i^2/\lambda$, expanding and rearranging) becomes 
$$0 \le 8i^2 + \lambda^3i + 4\lambda^2i + 8\lambda i + 2\lambda^2 + 4\lambda.$$ 
This clearly holds for all $\lambda > 0$ and $i \in \N$, and so the theorem follows.
\end{proof}

Now, combining Theorem~\ref{thm:f:sufficient} and Lemma~\ref{lem:f:recursion}, we can deduce from (sufficiently strong) bounds on $f_k$, for any $k \in \N$, that the uncovered set is non-empty. To do so, observe first that, given~$f_{i-1}$, the bound on $f_i$ given by Lemma~\ref{lem:f:recursion} is just a function of a \emph{single\/}~$\delta_i$. Elementary calculus shows that the optimal choice of $\delta_i$ occurs when 
\begin{equation}\label{e:deltachoice}
 \delta_i 
 = \frac{1+a_i}{1+\sqrt{1+a_i(1+a_i)/(b_if_{i-1})}}.
\end{equation}
This expression for $\delta_i$ allows for very fast numerical computation of the bounds on the~$f_i$.


\begin{table}
\[\begin{array}{c|c|c}
 k&p_k&g_k \ge \\
 \hline
 2&3&1.260997\\
 3&5&3.007888\\
 4&7&5.860938\\
 5&11&9.032082\\
 6&13&13.30344\\
 7&17&17.99687\\
 8 & 19  & 23.90973 \\
 9 & 23  & 30.38722 \\
10 & 29 & 36.72372 \\
100 & 541 & 1691.365 \\
1000 & 7919  & 42420.78 \\
10000 & 104729  & 802133.7 \\
51000 & 625187 & 5821999
\end{array}\]
\medskip
\caption{Upper bounds on $f_k$ that ensure that the system does not cover the integers. All the bounds are rounded down in the last decimal digit. At each stage in the calculation, $f_k$ was increased by a factor of $1+10^{-15}$ to account for rounding errors in the floating point arithmetic. The choice of $51000$ is needed to make $p_k > 616000$, as used in Section~\ref{sec:minmod}, below.}\label{t:1}
\end{table}

Let us therefore, for each $i \in \N$, define $g_i$ to be the largest value of $f_i(\cA)$ such that, by repeatedly applying the recursion~\eqref{e:f_i} with $\delta_j$ given by~\eqref{e:deltachoice}, we eventually satisfy the conditions of Theorem~\ref{thm:f:sufficient} for some $k \ge 10$. In Table~\ref{t:1} we list the bounds on $g_k$ given by performing this calculation, which was implemented as follows: starting with a potential value of $f_3$, we ran the iteration given in~\eqref{e:f_i} using the value of $\delta_i$ given in~\eqref{e:deltachoice} until either the conditions of Theorem~\ref{thm:f:sufficient} were satisfied, or the condition $b_if_{i-1} < \delta_i(1-\delta_i)$ failed. The optimal value $g_3$ of $f_3$ was determined by binary chop, and the other bounds $g_k$ were read off by taking the largest successful $f_3$ and listing the corresponding bounds on $f_k$.

Let us state, for future reference, the conclusion of this section.

\begin{corollary}\label{cor:gk}
If $f_k(\cA) \le g_k$ for some $k \in \N$, then the system of arithmetic progressions $\cA$ does not cover\/~$\Z$.
\end{corollary}

We remark that the bound on $g_1$ given by our sieve is less than $1$, and for this reason we are unable to resolve the Erd\H{o}s--Selfridge problem.

\section{The Erd\H{o}s--Selfridge Problem}\label{sec:ES}

In this section we will prove Theorem~\ref{thm:235}, which is a simple consequence of the method described in the previous section. First, however, let us show how to prove the following (only slightly weaker) theorem, which was first proved by Hough and Nielsen~\cite{HN}.

\begin{theorem}\label{thm:23}
Let\/ $\cA$ be a finite collection of arithmetic progressions with distinct moduli, none of which is divisible by $2$ or $3$. Then $\cA$ does not cover the integers. 
\end{theorem}

\begin{proof}
Set $i_0 = 2$ and $\kappa = 1$, note that $\mu_1 = \mu_2 = 1$, since there are no moduli divisible by 2 or 3, and recall from~\eqref{kappa:sufficient} that this is a valid choice of $\kappa$, by Theorem~\ref{prop:moments}. Recalling from~\eqref{fk:def} that $f_2 = \kappa / \mu_2$, and using Table~\ref{t:1}, we see that $f_2 = 1 < 1.26 < g_2$, and hence, by Corollary~\ref{cor:gk}, the system $\cA$ does not cover the integers.
\end{proof}

We remark that we did not actually need the full strength of Corollary~\ref{cor:gk} to prove this theorem; in fact, we could have just run our sieve with $\delta_1 = \cdots = \delta_n = 1/4$, say, and applied Theorems~\ref{thm:general} and~\ref{prop:moments}, together with Theorem~\ref{thm:f:sufficient}. In order to prove Theorem~\ref{thm:235}, we will need a slightly more complicated version of the proof above.



\begin{proof}[Proof of Theorem~\ref{thm:235}]
By Theorem~\ref{thm:23}, we may assume that $Q = \lcm(D)$ satisfies $Q = 3Q'$, where $Q'$ is not divisible by $2$, $3$ or $5$. Observe that $\mu_1 = 1$, and that $\mu_3 = \mu_2 \ge 2/3$, by the (trivial) first moment bound $M_2^{(1)}\le 1/3$, and since there are no moduli divisible by~2 or~5. Set $i_0 = 3$, $\delta_1 = \delta_2 = \delta_3 = 0$, and $\kappa = 2$. To see that this is a valid choice of $\kappa$, we need to improve~\eqref{kappa:sufficient} slightly, using the fact that $3^2 \nmid Q$. To be precise, in the proof of Lemma~\ref{lem:sums}, in the last expression in~\eqref{eq:sum:of:powers}, when $j = 2$ we only need to include the term $t = 1$ in the sum. Keeping the rest of the proof of Theorem~\ref{prop:moments} the same, this implies that
\[
M_i^{(2)}  \le \frac{1}{(p_i-1)^2} \bigg( 1 + \frac{3}{(1-\delta_2)p_2} \bigg) \prod_{3 < j < i} \bigg( 1 + \frac{3p_j-1}{(1-\delta_j)(p_j-1)^2} \bigg),
\]
and hence~\eqref{kappa:def:assumption} holds with $\kappa = 1 + 3/p_2 =2$, as claimed. Using Table~\ref{t:1}, it follows that
\[
 f_3 = \frac{\kappa}{\mu_3} \le 3 < 3.007 \le g_3,
\]
and therefore, by Corollary~\ref{cor:gk}, the system $\cA$ does not cover the integers.
\end{proof}

\section{The Minimum Modulus Problem}\label{sec:minmod}

In this section we improve the bound on the minimum modulus given in~\cite[Theorem~1]{H}.

\begin{theorem}
Let\/ $\cA$ be a finite collection of arithmetic progressions with distinct moduli $d_1,\dots,d_k \ge 616000$. Then $\cA$ does not cover the integers. 
\end{theorem}

\begin{proof}
We apply Theorem~\ref{thm:general}, using the first moment $M^{(1)}_i$ and setting $\delta_i=0$ for $i\le 51$ (note that $p_{51} = 233$). After the first 51 primes we have
\begin{equation}\label{e:mmbound1}
 \mu_{51} \ge 1 - \sum_{\substack{d\ge 616000 \\ p_{51}\text{-smooth}}}\frac{1}{d} \, \ge \, 0.654258
\end{equation}
and 
$$f_{51} \le \frac{1}{\mu_{51}} \prod_{j \le 51} \bigg( 1 + \frac{3p_j-1}{(p_j-1)^2} \bigg) \le \, 886.56.$$ 
For $51 < i \le 51000$ we apply the second moment bound using
\begin{equation}\label{e:mmbound2}
M^{(2)}_i\le \hat M^{(2)}_i:= \sum_{\substack{m_1p_i^{j},m_2p_i^{k}\ge 616000 \\ m_1,m_2\ p_{i-1}\text{-smooth}}} p^{-j-k} \cdot\frac{\nu\big( \lcm(m_1,m_2) \big)}{\lcm(m_1,m_2)}.
\end{equation}
The values of $\delta_i$ were not optimized, but instead defined by the following equation
\[
\delta_i = \bigg( 1 - \frac{1}{\sqrt{p_i}} \bigg) \cdot \frac{1+a_i}{1+\sqrt{1+ 4\hat \mu_i a_i(1+a_i)/\hat M_i^{(2)}}},
\]
which is based on~\eqref{e:deltachoice}, but with $b_i f_{i-1}$ replaced by the bound on $M_i^{(2)}/(4\mu_i)$ implied by~\eqref{e:mmbound2}. The (rather arbitrary) factor of $(1-1/\sqrt{p_i})$ was included to improve the bounds obtained, and  $\hat\mu_i$ is the lower bound on $\mu_i$ defined inductively by
\[
\hat\mu_i= \hat\mu_{i-1} - \frac{\hat M_i^{(2)}}{4\delta_i(1-\delta_i)}.
\]
Finally, after processing $p_{51000} = 625187 > 616000$ we calculated the bound $f_{51000} \le 5589593$ from~\eqref{fk:def} using $i_0=0$ and $\kappa = 1$. This is less than the bound $g_{51000}$ given in Table 1, and hence, by Corollary~\ref{cor:gk}, the system $\cA$ does not cover the integers.

It only remains to describe an efficient way of calculating the expressions \eqref{e:mmbound1}
and~\eqref{e:mmbound2}. For~\eqref{e:mmbound1} we note that the sum of $1/d$ over all $p_i$-smooth
$d$ is given by the product
\[
 \sum_{d\ p_{i}\text{-smooth}}\frac{1}{d} \, = \, \prod_{j\le i} \bigg( 1 + \frac{1}{p_j - 1} \bigg)
\]
and the sum in~\eqref{e:mmbound1} can then be calculated by subtracting the finite sum of $1/d$
over all $p_i$-smooth $d<616000$. For~\eqref{e:mmbound2} the procedure is somewhat more complicated.
First we define
\[
 \Theta_i(s,t):=\sum_{\substack{m_1 \ge s,\,m_2 \ge t \\ m_1,m_2\ p_{i}\text{-smooth}}}\frac{\nu\big( \lcm(m_1,m_2) \big)}{\lcm(m_1,m_2)},
\]
which can be calculated inductively using the identity
\[
\Theta_i(s,t)=\Theta_{i-1}(s,t) + \frac{1}{1 - \delta_i} \sum_{j,k\ge0,\, j + k > 0} p_i^{-\max\{j,k\}} \cdot \Theta_{i-1}\big(\lceil s/p_i^j\rceil,\lceil t/p_i^k\rceil\big),
\]
which, despite its appearance, can be calculated as a finite sum. Indeed, $\lceil s/p_i^j\rceil=1$
for sufficiently large~$j$, and so there are only finitely many terms $\Theta_{i-1}(s',t')$ that occur,
and these are multiplied (when $s'$ or $t'=1$) by geometric series that can be summed exactly.
Finally, the calculation of
\[
 M^{(2)}_i\le \sum_{j,k\ge1} p_i^{-j-k} \cdot \Theta_{i-1}\big(\lceil K/p_i^j\rceil,\lceil K/p_i^k\rceil\big),
\]
where $K = 616000$, can similarly be reduced to a finite sum.
\end{proof}

\section{Schinzel's Conjecture}\label{Schinzel:sec}

In this section we will use the method of Section~\ref{sec:general:method} to prove Schinzel's Conjecture~\cite{Sch}, which we restate here for convenience.

\begin{theorem}
If\/ $1<d_1<d_2<\dots<d_k$ are the moduli of a finite collection of arithmetic progressions that covers the integers, then\/ $d_i\mid d_j$ for some\/ $i<j$.
\end{theorem}

We will argue by contradiction, assuming that we have a set $\{d_1,\dots,d_k\}$ of moduli of a covering system of the integers that forms an antichain under divisibility. We call an antichain of natural numbers $p$-{\em smooth\/} if all its elements are $p$-smooth, i.e., have no prime factor greater than~$p$. In order to apply our sieve, we will need the following three simple lemmas about $5$-smooth antichains. 

\begin{lemma}\label{antichains:first}
If\/ $D$ is a $5$-smooth antichain containing no prime power, then 
$$\sum_{d \in D} \frac{1}{d} \, \le \, \frac13,$$ 
with equality if and only if\/ $D = \{6,10,15\}$.
\end{lemma}

\begin{proof}
Suppose first that $D'$ is a $3$-smooth antichain (possibly containing a prime power), and observe that 
$$D' = \big\{ 2^{a_1}3^{b_1},2^{a_2}3^{b_2},\dots,2^{a_k}3^{b_k} \big\}$$ 
for some $a_1>a_2>\dots>a_k$ and $b_1<b_2<\dots<b_k$. We claim that $\sum_{d \in D'} \frac{1}{d}$ is maximized when $D'$ is `compressed', that is, $a_i = a_{i+1} + 1$ and $b_{i+1} = b_i + 1$ for all~$i \in [k-1]$. Indeed, if $a_i > a_{i+1}+1$ then we can reduce $a_i$, and if $b_{i+1} > b_i+1$ then we can reduce $b_{i+1}$, in each case increasing $\sum_{d \in D'} \frac{1}{d}$ while maintaining the property that $D'$ is an antichain.

Now, let us write the 5-smooth antichain $D$ as a union of sets of the form $\{ 5^i d : d \in D_i \}$, where each $D_i$ is a (possibly empty) 3-smooth antichain. Since $D$ contains no prime power, neither can $D_0$, so $\sum_{d \in D_0} \frac{1}{d}$ is maximized when $D_0 = \{2^a\cdot3,2^{a-1}\cdot3^2,\dots,2\cdot3^a\}$. In this case a simple calculation shows that $\sum_{d \in D_0} \frac{1}{d} = 2^{-a} - 3^{-a}$, which attains a unique maximum when $a=1$. Thus $\sum_{d \in D_0} \frac{1}{d} \le \frac{1}{6}$, with equality if and only if $D_0 = \{6\}$.

Next, observe that $1 \not\in D_i$ for every~$i \ge 1$, and therefore $\sum_{d \in D_i} \frac{1}{d}$ is maximized when $D_i = \{2^a,2^{a-1}3^1,\dots,3^a\}$, with $a \ge 1$. In this case we have $\sum_{d \in D_i} \frac{1}{d} = 6(2^{-1-a}-3^{-1-a})$ and the unique extremal case is $a=1$, so $\sum_{d \in D_i} \frac{1}{d} \le \frac{5}{6}$, with equality if and only if $D_i = \{2,3\}$. If we additionally assume that $D_i \ne \{2,3\}$, then $\sum_{d \in D_i} \frac{1}{d} \le \frac{11}{18}$, with the unique maximum occurring for the (uncompressed) antichain $D_i = \{2,9\}$.

Finally, if $D_1 = \{2,3\}$ then $D_i = \emptyset$ for every $i > 1$, and therefore
$$\sum_{d \in D} \frac{1}{d} \, = \, \sum_{d \in D_0} \frac{1}{d} \, + \frac{1}{5}\sum_{d \in D_1} \frac{1}{d} \, \le \, \frac{1}{6} + \frac{1}{5}\cdot\frac{5}{6} \, = \, \frac{1}{3},$$ 
with equality only when $D = \{6,10,15\}$. On the other hand, if $D_1\ne \{2,3\}$ then
$$\sum_{d \in D} \frac{1}{d} \, = \, \sum_{i = 0}^\infty \frac{1}{5^i} \sum_{d \in D_i} \frac{1}{d} \, \le \, \frac{1}{6} + \frac{11}{5 \cdot 18} + \sum_{i = 2}^\infty \frac{1}{5^{i-1} \cdot 6} \, = \, \frac{119}{360}  \, < \, \frac{1}{3},$$ 
as required.
\end{proof}

\begin{lemma}\label{antichains:second}
If\/ $A$ and $B$ are two $3$-smooth antichains, then
\begin{equation}\label{eq:antichains:second}
\sum_{a \in A} \sum_{b \in B} \frac{1}{\lcm(a,b)} \le \,\frac{31}{36},
\end{equation}
except in the cases $A=B=\{1\}$ and $A=B=\{2,3\}$. 
\end{lemma}

Note that the sum in~\eqref{eq:antichains:second} is equal to $1$ if $A=B=\{1\}$, and to $7/6$ if $A=B=\{ 2,3 \}$.

\begin{proof}
Suppose first that $|A| \le |B|$, and that $B = \{2^{k-1},2^{k-2}3^1,\dots,3^{k-1}\}$ for some $k \ge 5$. Then
$$\sum_{a \in A} \sum_{b \in B} \frac{1}{\lcm(a,b)} \, \le \, \sum_{a \in A} \sum_{b \in B} \frac{1}{b} \, = \, |A| \sum_{b \in B} \frac{1}{b} \, = \, 6k \big( 2^{-k} - 3^{-k} \big) \, < \, \frac{31}{36}.$$
However, if $|B| = k \ge 5$ but $B \ne \{2^{k-1},2^{k-2}3^1,\dots,3^{k-1}\}$, then by compressing (as in the proof of Lemma~\ref{antichains:first}) we may increase the left-hand side of~\eqref{eq:antichains:second}, so we are also done in this case. The lemma therefore reduces to a finite check of families with $\max\{ |A|,|B| \} \le 4$, and in fact (using compression once again) it is sufficient to consider the antichains $\{1\}$, $\{2,3\}$, $\{2,9\}$, $\{3,4\}$, $\{4,6,9\}$, and $\{8,12,18,27\}$. The lemma now follows from a trivial case analysis, which can be done by hand.
\end{proof}

\begin{lemma}\label{antichains:third}
 If\/ $A$ and $B$ are two $5$-smooth antichains, then 
 $$\sum_{a \in A} \sum_{b \in B} \frac{1}{\lcm(a,b)} \le \,1.7,$$ 
with equality if and only if $A=B=\{2,3,5\}$.
\end{lemma}

\begin{proof}
We decompose $A$ and $B$ as a union of sets $5^i\cdot A_i$ and $5^j\cdot B_j$, where $A_i$ and $B_j$ are 3-smooth antichains, as in the proof of Lemma~\ref{antichains:first}. Suppose first that there is no pair $(i,j)$ with $A_i = B_j = \{1\}$ or $A_i = B_j = \{2,3\}$. Then, by Lemma~\ref{antichains:second} and~\eqref{eq:sum:odds:over:powers}, we have 
\begin{align*}
\sum_{a \in A} \sum_{b \in B} \frac{1}{\lcm(a,b)} & \, = \, \sum_{i = 0}^\infty \sum_{j = 0}^\infty \frac{1}{5^{\max\{i,j\}}} \sum_{a \in A_i} \sum_{b \in B_j} \frac{1}{\lcm(a,b)}\\
& \, \le \, \frac{31}{36} \bigg( 1 + \frac{3}{5} + \frac{5}{5^2} + \frac{7}{5^3} + \cdots \bigg) = \frac{31}{36} \cdot \frac{15}{8} < 1.7 - \frac{1}{12}.
\end{align*}
Next, suppose that $A_i = B_j = \{1\}$ for some pair $(i,j)$, and observe that $A_{i'} = B_{j'} = \emptyset$ for every $i' > i$ and $j' > j$, so the pair $(i,j)$ is unique. If there is no pair $(s,t)$ with $A_s = B_t = \{2,3\}$, then the bound above increases by at most $(1 - \tfrac{31}{36})5^{-\max\{i,j\}}$, and this is less that $\frac{1}{12}$ if $\max\{i,j\}\ge1$. On the other hand, if $A_0 = B_0 = \{1\}$, then (since $A_i = B_i = \emptyset$ for all $i > 0$) we have $A = B = \{1\}$, and so $\sum_{a \in A} \sum_{b \in B} \frac{1}{\lcm(a,b)} = 1 < 1.7$.

We may therefore assume that $A_i = B_j = \{2,3\}$ for some pair $(i,j)$, which implies that $A_{i'},B_{j'} \subseteq \{1\}$ for every $i' > i$ and $j' > j$, and (as above) at most one of the sets in each sequence is non-empty. The bound above increases by at most 
\[
 \bigg( \frac{7}{6} - \frac{31}{36} \bigg) \frac{1}{5^{\max\{i,j\}}} + \bigg( 1 - \frac{31}{36} \bigg) \frac{1}{5^{\max\{i,j\} + 1}} = \frac{1}{3} \cdot \frac{1}{5^{\max\{i,j\}}} < \frac{1}{12}
\]
if $\max\{i,j\} \ge 1$. However, if $A_0 = B_0 = \{2,3\}$, then it is easy to see that $\sum_{a \in A} \sum_{b \in B} \frac{1}{\lcm(a,b)}$ is maximized by taking $A = B = \{2,3,5\}$, and in that case it is equal to $1.7$.
\end{proof}

Having completed the easy preliminaries, we are ready to prove Schinzel's conjecture.

\begin{proof}[Proof of Theorem~\ref{thm:Schinzel}.]
We first observe that we may assume that none of the moduli $d_i$ are prime powers. Indeed, we
may assume that the covering is minimal, so the removal of any $A_{d_i}$ results in a
set of progressions that do not cover $\Z$.
If $d_i=p^j$ for some prime $p$ and $j>0$, then the prime can appear at most to the $(j-1)$st
power in any other moduli. Thus the other progressions fail to cover some congruence class
mod $Q/p$, where $Q=\lcm\{d_1,\dots,d_k\}$. But this congruence class cannot be covered
by $A_{d_i}$ as $d_i\nmid Q/p$, a contradiction.

We now apply our sieve with $\delta_1 = \delta_2 = \delta_3 = 0$, so that $\Pr_i$ is equal to the uniform measure when processing the primes $p_1 = 2$, $p_2 = 3$ and $p_3 = 5$, and claim that $f_3 \le 2.55 < g_3$. Observe that, by Lemma~\ref{antichains:first}, the total measure of $B_1 \cup B_2 \cup B_3$ is at most $1/3$. Now we improve the bound on $\E[\alpha_i(x)^2]$ for $i \ge 4$ as follows. By Theorem~\ref{prop:moments}, we have 
$$M_i^{(2)} \le \sum_{j_1,j_2 \ge 1} \frac{1}{p_i^{j_1+j_2}} \sum_{m_1,m_2\in S_i}\frac{\nu\big( \lcm(m_1,m_2) \big)}{\lcm(m_1,m_2)} \sum_{a \in D(m_1,j_1)} \sum_{b \in D(m_2,j_2)} \frac{1}{\lcm(a,b)},$$
where $S_i$ is the set of integers whose prime factors all lie between $7$ and $p_{i-1}$,
and for each $m \in S_i$ and $j \ge 1$, we define 
$$D(m,j) = \big\{ a : m p_i^j a \in D \text{ and $a$ is 5-smooth} \big\}.$$ 
Since $D(m,j)$ is a 5-smooth antichain, it follows from Lemmas~\ref{lem:sums} and~\ref{antichains:third} that
$$M_i^{(2)} \le 1.7 \sum_{j_1,j_2\ge1}\frac{1}{p_i^{j_1+j_2}}\sum_{m_1,m_2\in S_i}\frac{\nu\big( \lcm(m_1,m_2) \big)}{\lcm(m_1,m_2)} \, = \, \frac{1.7}{(p_i-1)^2}\prod_{j=4}^{i-1}\bigg( 1 + \frac{3p_j-1}{(1-\delta_j)(p_j-1)^2} \bigg).$$
Therefore, setting $i_0 = 3$ and $\kappa = 1.7$, it follows that~\eqref{kappa:def:assumption} holds. Hence, recalling from above that $\mu_3 \ge 2/3$, we obtain $f_3 = 1.7 \cdot 3/2 = 2.55 < g_3$ (see Table~\ref{t:1}), so, by Corollary~\ref{cor:gk}, the system $\cA = \{A_d : d \in D\}$ does not cover the integers.
\end{proof}

\section{Constructions}\label{sec:construc}

In this section we will provide constructions of families of arithmetic progressions that answer
(negatively) the question of Filaseta, Ford, Konyagin, Pomerance and Yu~\cite{FFKPY} mentioned
in the Introduction, and show that Theorem~\ref{thm:uncoveredDensity} is not far from best possible.
To be precise, we will prove the following two theorems.

\begin{theorem}\label{thm:CEtoFFKPYQuestion}
 For every\/ $M > 0$ and\/ $\eps > 0$, there exists a finite collection of arithmetic progressions\/
 $A_1, \dots, A_k$ with distinct moduli\/ $d_1,\dots,d_k\ge M$, such that
 \begin{equation}\label{eq:sum:less:than:one}
  \sum_{i=1}^k\frac{1}{d_i} < 1
 \end{equation}
 and the density of the uncovered set\/ $R = \Z \setminus \bigcup_{i=1}^k A_i$ is less than\/~$\eps$.
\end{theorem}

Note that the bound~\eqref{eq:sum:less:than:one} is (obviously) best possible; we remark that the
moduli in our construction will moreover be square-free. Our second theorem shows that the function
$(\log p)^{3+\eps}$ in the statement of Theorem~\ref{thm:uncoveredDensity} cannot be replaced by a constant.

\begin{theorem}\label{thm:ThmAlmostSharp}
 For every\/ $\lambda > 0$, there exists\/ $C = C(\lambda) > 0$ such that the following holds.
 Let\/ $\mu$ be the multiplicative function defined by
 \[
 \mu(p^i) = 1 + \frac{\lambda}{p}
 \]
 for all primes\/ $p$ and integers\/~$i \ge 1$. For every\/ $M > 0$ and\/ $\eps > 0$, there exists
 a finite collection of arithmetic progressions\/ $A_1, \dots, A_k$ with distinct square-free
 moduli\/ $d_1,\dots,d_k\ge M$, such that
 \[
  \sum_{i=1}^k \frac{\mu(d_i)}{d_i} \le C,
 \]
 and the density of the uncovered set\/ $R = \Z \setminus \bigcup_{i=1}^k A_i$ is at most\/~$\eps$.
\end{theorem}

The proof of Theorem~\ref{thm:CEtoFFKPYQuestion} is relatively simple, while the proof of
Theorem~\ref{thm:ThmAlmostSharp} will require somewhat more work.

\begin{proof}[Proof of Theorem~\ref{thm:CEtoFFKPYQuestion}]
We will choose a collection $P_1, \dots, P_N$ of disjoint sets of primes, and define
\[
 Q_i := \prod_{j \le i} \prod_{p\in P_j}  p \qquad \text{and} \qquad
 D_i := \big\{ p \cdot Q_{i-1} : p \in P_i \big\}
\]
for each $i \in [N]$, where $Q_0 := 1$. We will show that, for a suitable choice of $P_1, \dots, P_N$,
the set $D = D_1 \cup \cdots \cup D_N$ has the following properties:
\begin{equation}\label{eq:sum:of:ds:almost:one}
 \sum_{d \in D} \frac{1}{d} \le 1 + \frac{\eps}{3},
\end{equation}
and there exists a collection of arithmetic progressions, with distinct moduli in~$D$,
such that the uncovered set has density at most $\eps / 3$. By removing a few of the progressions
from this family, we will obtain the claimed collection $A_1, \dots, A_k$.

We construct the sets $P_1, \dots, P_N$ of primes as follows. First, let us fix some positive
constants $c_0$, $c$ and $\delta$ such that
\[
 c_0 := 1 + \frac{\eps}{3}, \qquad \text{and} \qquad c=\frac{\delta}{1-e^{-\delta}}\in(1,c_0).
\]
Indeed, $\delta/(1-e^{-\delta})$ is a continuous increasing function of $\delta$ which tends to 1 as $\delta\to0$,
so for sufficiently small $\delta$ we have $1<\delta/(1-e^{-\delta})<c_0$.
Assume (without loss) that $M > 3/\eps$, and choose $N$ sufficiently large so that $e^{-\delta N} < \eps/3$.
Now let $P_1$ be any set of primes such that $p \ge M$ for every $p \in P_1$, and
\[
 \delta \le \sum_{p \in P_1} \frac{1}{p} \le \delta + \frac{c_0 - c}{N}.
\]
In general, if we have already constructed $P_1,\dots,P_j$, then let $P_{j+1}$ be any set of primes
(disjoint from $P_1 \cup \cdots \cup P_j$) such that $p \ge M$ for every $p \in P_{j+1}$, and
\begin{equation}\label{eq:choosing:Pjs}
 \delta e^{-\delta j} \le \sum_{p \in P_{j+1}} \frac{1}{p \cdot Q_j}
 \le \delta e^{-\delta j} + \frac{c_0 - c}{N}.
\end{equation}
As the sum $\sum 1/p$ over prime $p$ diverges, it is clear that sets $P_1,\dots,P_N$ exist
with these properties.  
It follows that
\[
 \sum_{d \in D} \frac{1}{d} = \sum_{j = 0}^{N-1} \sum_{p \in P_{j+1}} \frac{1}{p \cdot Q_j}
 \le \sum_{j = 0}^{N-1} \bigg( \delta e^{-\delta j} + \frac{c_0-c}{N} \bigg)
 \le c_0 = 1 + \frac{\eps}{3},
\]
where in the final inequality we used the identity
$\sum_{j = 0}^\infty \delta e^{-\delta j}  = \delta \big( 1 - e^{-\delta} \big)^{-1} = c$.

Now, to construct the arithmetic progressions, simply choose (for each $d \in D$ in turn) any
arithmetic progression with modulus $d$ that has at least the expected intersection with the
(as yet) uncovered set. To be more precise, for each $j \in [N]$ let $\cA_j$ denote the
collection of arithmetic progressions whose modulus lies in $D_j$, and write $\eps_j :=\Pr_0(R_j)$
for the density of the uncovered set $R_j := \Z_{Q_j} \setminus \bigcup_{A \in \cB_j} A$, where
$\cB_j := \cA_1\cup \dots\cup \cA_j$. Now observe that if $d = p \cdot Q_j \in D_{j+1}$,
then there are $p \cdot |R_j|$ congruence classes mod $d$ that completely cover~$R_j$.
Hence there is a congruence class that covers at least a fraction $1/(p \cdot |R_j|)$ of the
as yet uncovered set. It follows that
\[
 \eps_{j+1} \le \eps_j \prod_{p \in P_{j+1}} \bigg( 1 - \frac{1}{p \cdot |R_j|}\bigg)
 \le \eps_j \cdot \exp\bigg( - \frac{1}{|R_j|} \sum_{p \in P_{j+1}} \frac{1}{p}\bigg)
\]
for each $0 \le j \le N - 1$ (where $\eps_0 := 1$), and hence, by~\eqref{eq:choosing:Pjs},
\[
 \eps_{j+1} \le \eps_j \cdot \exp\bigg( - \frac{Q_j}{|R_j|} \cdot \delta e^{-\delta j} \bigg)
 = \eps_j \cdot \exp\bigg( - \frac{\delta e^{-\delta j}}{\eps_j} \bigg).
\]
It now follows immediately by induction that $\eps_j \le e^{-\delta j}$ for every $j \in [N]$,
and in particular $\eps_N \le e^{-\delta N} \le \eps / 3$, by our choice of~$N$.

We have therefore constructed a collection of arithmetic progressions whose set $D$ of
(distinct) moduli satisfies~\eqref{eq:sum:of:ds:almost:one}, and whose uncovered
set has density at most $\eps / 3$. To complete the construction, simply choose a maximal
subset $D' \subset D$ such that $\sum_{d \in D'} \frac{1}{d} < 1$, and observe that the density
of the set uncovered by $\big\{ A_d : d \in D' \big\}$ is at most
\[
 \frac{\eps}{3} + \sum_{d \in D \setminus D'} \frac{1}{d}
 \le \frac{\eps}{3} + \frac{\eps}{3} + \frac{1}{M} \le \eps,
\]
as required.
\end{proof}

The proof of Theorem~\ref{thm:ThmAlmostSharp} is similar to that of Theorem~\ref{thm:CEtoFFKPYQuestion},
but the details are somewhat more complicated.

\begin{proof}[Proof of Theorem~\ref{thm:ThmAlmostSharp}]
We will again choose a large collection of disjoint sets of primes, but this time we will arrange
them in a tree-like structure, and our common differences will be formed by taking products of
certain subsets of the primes along paths in the tree. To begin, let us choose $t > 2$ sufficiently
large so that $t^\lambda < e^{t-3}$, and let $P_1$ be a set of primes such that $p \ge M$ for every
$p \in P_1$, and
\[
 t - 1 \le \prod_{p\in P_1} \bigg( 1 + \frac{1}{p} \bigg) \le t.
\]
This is possible as the product $\prod_p (1+1/p)$ over all primes is infinite, and by taking only
primes greater than $t$ we can ensure that some finite product lands in $[t-1,t]$. Set
\[
 Q_1 := \prod_{p \in P_1} p \qquad \text{and} \qquad
 D_1 := \big\{ d > 1 : d \mid Q_1 \big\},
\]
and choose a collection of arithmetic progressions $\cA_1 = \{ a_d + d\Z : d \in D_1 \}$
so as to minimize the density of the uncovered set $R_1 := \Z_{Q_1} \setminus\bigcup_{A \in \cA_1}A$.
As in the previous proof, this can be done (greedily) so that
\[
 \Pr_0(R_1) \le \prod_{d \in D_1} \bigg( 1 - \frac{1}{d} \bigg)
 \le \exp\bigg( - \sum_{d \in D_1} \frac{1}{d} \bigg)
 \le \exp\bigg( 1 - \prod_{p \in P_1} \bigg( 1 + \frac{1}{p} \bigg) \bigg)
 \le e^{2-t}.
\]
Now, for each $x_1 \in R_1$, let $P_2^{x_1}$ be a set of new primes (i.e., disjoint for each choice of~$x_1$,
and disjoint from~$P_1$) such that $p \ge M$ for every $p \in P_2^{x_1}$, and
\[
 t - 1 \le \prod_{p \in P_2^{x_1}} \bigg( 1 + \frac{1}{p} \bigg) \le t.
\]
Set
\[
 Q_2^{x_1} := \prod_{p \in P_2^{x_1}} p \qquad \text{and} \qquad
 D_2^{x_1} := \big\{ d \cdot Q_1 : d > 1 \text{ and } d \mid Q_2^{x_1} \big\},
\]
and choose a collection of arithmetic progressions $\cA_2^{x_1} = \{ a_d + d\Z : d \in D_2^{x_1}\}$
so as to minimize the density of the uncovered set
$$R_2^{x_1} := \Big\{ (x_1,y) : y \in \Z_{Q_2^{x_1}} \Big\} \setminus\bigcup_{A \in \cA_2^{x_1}} A,$$
where (as usual) each $A \in \cA_2^{x_1}$ is viewed as a subset of $\Z_{Q_1} \times \Z_{Q_2^{x_1}}$. Note that
\[
 \Pr_0\big( R_2^{x_1} \big) \le \frac{1}{Q_1} \prod_{dQ_1 \in D_2^{x_1}} \bigg( 1 - \frac{1}{d} \bigg)
 \le \frac{1}{Q_1} \cdot \exp\bigg( 1 - \prod_{p \in P_2^{x_1}} \bigg( 1 + \frac{1}{p} \bigg) \bigg)
 \le \frac{e^{2-t}}{Q_1},
\]
and hence, setting $R_2 := \bigcup_{x_1 \in R_1} R_2^{x_1}$ and summing over $x_1 \in R_1$, we have
\[
 \Pr_0(R_2) \le |R_1| \cdot \frac{e^{2-t}}{Q_1} = \Pr_0(R_1) \cdot e^{2-t} \le e^{2(2-t)}.
\]

To describe a general step of this construction, suppose that we have already defined the tree of primes
and progressions to depth $i-1$, and for each $x_1 \in R_1$, $x_2\in R_2^{x_2}$, \dots,
$x_{i-1} \in R_{i-1}^{x_1,\dots,x_{i-2}}$, choose a set $P_i^{x_1,\dots,x_{i-1}}$ of new primes
(disjoint from all previously chosen sets) such that $p \ge M$ for every $p \in P_i^{x_1,\dots,x_{i-1}}$, and
\[
 t - 1 \le \prod_{p \in P_i^{x_1,\dots,x_{i-1}}} \bigg( 1 + \frac{1}{p} \bigg) \le t.
\]
Set
\[
 Q_i^{x_1,\dots,x_{i-1}} := \prod_{p \in P_i^{x_1,\dots,x_{i-1}}} p
\]
and
\[
 D_i^{x_1,\dots,x_{i-1}} := \big\{ d \cdot Q_1 \cdot Q_2^{x_1} \cdots
 Q_{i-1}^{x_1,\dots,x_{i-2}} : d > 1 \text{ and } d \mid Q_i^{x_1,\dots,x_{i-1}} \big\},
\]
and choose a collection of arithmetic progressions
$\cA_i^{x_1,\dots,x_{i-1}} = \{ a_d + d\Z : d \in D_i^{x_1,\dots,x_{i-1}} \}$ so as to minimize
the density of the uncovered set
$$R_i^{x_1,\dots,x_{i-1}} := \Big\{ (x_{i-1},y) : y \in \Z_{Q_i^{x_1,\dots,x_{i-1}}} \Big\} \setminus\bigcup_{A \in \cA_i^{x_1,\dots,x_{i-1}}} A,$$
where, as before, each $A \in \cA_i^{x_1,\dots,x_{i-1}}$ is viewed as a subset of $\Z_{Q_1} \times \cdots \times \Z_{Q_i^{x_1,\dots,x_{i-1}}}$. Setting $R_i := \bigcup_{x_1 \in R_1}\cdots \bigcup_{x_{i-1} \in R_{i-1}^{x_1,\dots,x_{i-2}}} R_i^{x_1,\dots,x_{i-1}}$,
and repeating the calculation above, we obtain
\[
 \Pr_0\big( R_i^{x_1,\dots,x_{i-1}} \big)
 \le \frac{1}{Q_1 \cdots Q_{i-1}^{x_1,\dots,x_{i-2}}}
  \prod_{dQ_1 \cdots Q_{i-1}^{x_1,\dots,x_{i-2}} \in D_i^{x_1,\dots,x_{i-1}}}\bigg( 1 - \frac{1}{d} \bigg)
 \le \frac{e^{2-t}}{Q_1 \cdots Q_{i-1}^{x_1,\dots,x_{i-2}}},
\]
and therefore
\[
 \Pr_0(R_i) \le \Pr_0(R_{i-1}) \cdot e^{2-t} \le e^{(2-t)i},
\]
by induction. Hence, defining $\cA_i$ to be the union of all $\cA_i^{x_1,\dots,x_{i-1}}$
and $\cA = \cA_1 \cup \cdots \cup \cA_n$, it follows that the uncovered set $R$ has density
\[
 \Pr_0(R) = \Pr_0(R_n) \le e^{(2-t)n} \to 0
\]
as $n \to \infty$.

It remains to show that
\[
 \sum_{d \in D} \frac{\mu(d)}{d} \le C
\]
for every $n \in \N$, where $D$ is the set of moduli of progressions in~$\cA$.
To prove this, observe first that
\[
 \mu(Q_1) = \prod_{p \in P_1} \bigg( 1 + \frac{\lambda}{p} \bigg)
 \le \prod_{p \in P_1} \bigg( 1 + \frac{1}{p} \bigg)^\lambda
 \le t^\lambda,
\]
and, assuming (as we may) that $M \ge \lambda$,
\[
 \sum_{d \in D_1} \frac{\mu(d)}{d}
 = \prod_{p \in P_1} \bigg( 1 + \frac{\mu(p)}p \bigg) - 1
 \le \prod_{p \in P_1} \bigg( 1 + \frac{1}{p} + \frac{\lambda}{p^2} \bigg)
 \le \prod_{p \in P_1} \bigg( 1 + \frac{1}{p} \bigg)^2 \le t^2.
\]
Similarly, we have
\[
 \sum_{d \in D_2^{x_1}} \frac{\mu(d)}{d}
 = \frac{\mu(Q_1)}{Q_1} \sum_{dQ_1 \in D_2^{x_1}} \frac{\mu(d)}{d}
 \le \frac{t^{\lambda + 2}}{Q_1}
\]
for each $x_1 \in R_1$, and, more generally,
\[
 \sum_{d \in D_i^{x_1,\dots,x_{i-1}}} \frac{\mu(d)}{d}
 \le \frac{\mu\big( Q_1 \cdots Q_{i-1}^{x_1,\dots,x_{i-2}} \big)}{Q_1 \cdots Q_{i-1}^{x_1,\dots,x_{i-2}}}
 \sum_{d \mid Q_i^{x_1,\dots,x_{i-1}}} \frac{\mu(d)}{d}
 \le \frac{t^{\lambda(i-1) + 2}}{Q_1 \cdots Q_{i-1}^{x_1,\dots,x_{i-2}}}.
\]
Hence, summing over $i \in [n]$ and sequences $x_1 \in R_1, \dots, x_{i-1} \in R_{i-1}^{x_1,\dots,x_{i-2}}$, we obtain
\[
 \sum_{d \in D} \frac{\mu(d)}{d}
 \le \sum_{i = 0}^{n-1} t^{\lambda i + 2} \cdot \Pr_0(R_i)
 \le \sum_{i = 0}^{n-1} t^{\lambda i + 2} \cdot e^{(2-t) i}
 < 2t^2,
\]
since $t^\lambda < e^{t - 3}$, as required.
\end{proof}

\section*{Acknowledgements}

This research was largely carried out during a one-month visit by the authors to IMT Lucca, and partly during visits by various subsets of the authors to IMPA and to the University of Memphis. We are grateful to each of these institutions for their hospitality, and for providing a wonderful working environment.

\end{document}